\documentclass[reqno,centertags,a4paper,10pt]{amsart}
\usepackage{amssymb,enumerate}
\usepackage{amsthm}
\usepackage{eucal}
\usepackage[notcite,notref]{showkeys}
\usepackage{fullpage}

\usepackage[utf8]{inputenc}		% Codificacao do documento (conversão automática dos acentos)

\usepackage{pslatex}			% Usa a fonte Latin Modern	

\usepackage{indentfirst}		% Indenta o primeiro parágrafo de cada seção.
\usepackage{color}				% Controle das cores
\usepackage{graphicx}			% Inclusão de gráficos
\usepackage{microtype} 			% para melhorias de justificação
\usepackage{amsmath}

\usepackage{hyperref}     %link das referencias aparecer no PDF

\usepackage[all]{xy}

\usepackage{pgf,tikz}

\usetikzlibrary{arrows}

\DeclareMathOperator{\supp}{supp}

\newcommand{\R}{\textnormal{I}\!\textnormal{R}}

\newcommand{\cjap}[1]{\langle {#1}\rangle}

\newcommand{\FT}{\mathcal{F}}

\newcommand{\IFT}{\mathcal{F}^{\!^{-1}}}

\def\smath#1{\text{\scalebox{.8}{$#1$}}}

\def\pow{^\smath}

\newcommand{\norma}[2]{\left\|{#1}\right\|_{\!_{#2}}}           %%%%norma   
\newcommand{\normaNp}[1]{\left\|{#1}\right\|_{\!\mathcal{N}\pow{-p/2}}}

\newcommand{\normaSp}[1]{\left\|{#1}\right\|_{\!\mathcal{S}\pow{-p/2 } } }
\newcommand{\normaXp}[1]{\left\|{#1}\right\|_{\!X_p\pow {-\frac{p}{2},-\frac{1}{2},1}}}
\newcommand{\normaYp}[1]{\left\|{#1}\right\|_{\!Y_p\pow {-p/2,-1/2}}}
\newcommand{\PN}{P_{N} }

\newcommand{\PNU}{P_{{N_1}} }

\newcommand{\PND}{P_{{N_2}} }

\newcommand{\phiN}{\varphi_{\!_{N}}} 

\newcommand{\phiL}{\varphi_{\!_{L}}} 

\newcommand{\QL}{Q_{L}}

\newcommand{\QLU}{Q_{L_1}}
\newcommand{\QLD}{Q_{L_2}}

\numberwithin{equation}{section}

\newtheorem{theorem}{Theorem}[section]
\newtheorem{proposition}[theorem]{Proposition}
\newtheorem{remark}[theorem]{Remark}
\newtheorem{lemma}[theorem]{Lemma}
\newtheorem{corollary}[theorem]{Corollary}

\begin{document}

	\title[Sharp  well-posedness and ill-posedness results for the generalized KdV-B equation on the real line]
	{Sharp   well-posedness  and ill-posedness results for dissipative KdV equations on the real line}
	
	\author[ X. Carvajal]
	{Xavier Carvajal}
\address{Xavier Carvajal}                  
\email{carvajal@im.ufrj.br}

\author[ P. Gamboa] {Pedro Gamboa}
	\address{Pedro Gamboa}                  
	\email{pgamboa@im.ufrj.br}
	
	\author[ R. Santos]
	{Raphael Santos	}
\address{Raphael Santos}                  
	\email{raphaelsantos@macae.ufrj.br}

	%\thanks{bla bla}

	
	\begin{abstract}
		
		This work is concerned about the Cauchy problem for the following generalized KdV- Burgers equation
		\begin{equation*}
		\left\{\begin{array}{l}
		\partial_tu+\partial_x^3u+L_pu+u\partial_xu=0,\\
		u(0,\,x)=u_0(x).
		\end{array}
		\right.
			\end{equation*} 
		where $L_p$ is a dissipative multiplier operator.  Using Besov-Bourgain Spaces, we establish a bilinear estimate and following the framework developed in \cite{MV} we prove sharp local and global well-posedness in the Sobolev spaces $H^{-p/2}(\R)$ and  ill-posedness in $H^s(\R)$ when $s<-p/2$, both when  $p \geq 2$. Also, we prove $C^2$-ill-posedness in $H^s(\R)$, for $s< 3/2-p/4$ and $0\leq p\leq 2$.
	\end{abstract}
	
	\maketitle
	
	\allowdisplaybreaks
	
	\section{Introduction}
	
	In this paper we study the well-posedness of the generalized Korteweg-De Vries- Burgers equation
	\begin{equation}\tag{g-KdV-B}\label{g-KdV-B}
	\partial_tu+\partial_x^3u+L_pu+u\partial_xu=0,\quad x\in\R,\; t\geq 0,
	\end{equation}
	where $u=u(t,x)$ is a real-valued function  and $\FT_x \{L_pu\}(t,\xi)=|\xi|^p\FT_x u(t,\xi)$, for $p\in \R^{+}$. When $p=2$ we have the well-known   KdV - Burgers equation. This equation arises in some different physical contexts as a model equation involving the effects of dispersion, dissipation and nonlinearity. When $p=1/2$ the related equation models the evolution of the free surface for shallow water waves damped by viscosity. For these models, see e.g. \cite{JON}, \cite{KM}, and \cite{OS}.
	
	The well-posedness for the equation \eqref{g-KdV-B} has been studied for many authors. In 2001, using the  Bourgain spaces, related only to the KdV equation (see e.g. \cite{bourgain} and \cite{ginibre}), and the bilinear estimate due to Kenig, Ponce and Vega (see \cite{KPV}), Molinet and Ribaud obtained the  local  and global well-posedness  in $H^s(\R)$, for $s>-3/4$ and $p>0$. In the particular case of $p=2$ (KdV-Burgers equation), they proved local and global well-posedness in $H^s(\R)$, for $s> -3/4 -1/24$ (see \cite{MR1}). In 2002, they improved the result when $p=2$, by using the Bourgain space but now, associated to the KdV-Burgers equation, getting local and global well-posedness in $H^s(\R)$, for $s>-1$ (see \cite{MR}). Also, in this paper they pointed out that the Cauchy problem \eqref{g-KdV-B}, with $0\leq p\leq 2$ is ill-posed in the homogeneous Sobolev
	space $\textrm{\.H}^s(\R)$ for $s<s_p$, where $s_p=(p-6)/2(4-p)$, and conjectured that $H^{s_p}(\R)$ is the critical Sobolev spaces and the Cauchy problem for \eqref{g-KdV-B} is
	well-posed in $H^s(\R)$ for $s>s_p$. 	In 2010, Xue and Hu proved the local well-posedness (l.w.p.) for the \eqref{g-KdV-B} in the homogeneous Sobolev spaces $\textrm{\.H}^s(\R)$, with $(p-6)/2(4-p)<s\leq 0$, when $0\leq p\leq 2$, giving a partial answer for this open problem (see \cite{XH}). In 2011, Vento proved local and global well-posedness for the \eqref{g-KdV-B} in $H^s(\R)$, for $s>s_p$ where
	\begin{equation}\label{V01}
	s_p=\left\{\begin{array}{l}
	-3/4,\ \ \ \ 0<p\leq1,\\
	-3/(5-p),\ \ \ \ 1<p\leq 2,
	\end{array}
	\right.
	\end{equation}
	improving the early results in  the case  $1<p<2$ (see \cite{V}). Also, in 2011 Molinet and Vento completes the result for the KdV - Burgers equation ($p=2$), using the Besov refinement of Bourgain's spaces. They obtained the sharp g.w.p. in $H^{-1}(\R)$ (see \cite{MV}). In 2016 Carvajal and Mahendra studied, among other things,  the well-posedness of the following dissipative versions of the generalized KdV equation

	\begin{equation}\label{XM01}
	\left\{\begin{array}{l}
	v_t + v_{xxx} + \eta Lv + (v^{2})_x = 0, \quad x\in \R,\ t\geq 0,\\
	v(x, 0) = v_0(x),
	\end{array}
	\right.
	\end{equation}
	 where $\eta>0$ and the linear operator $L$ is defined via the Fourier transform by $\FT_x\{Lf\}=\Phi(\cdot)\FT_xf$, where the symbol 
	 \begin{equation}\label{XM02}
	\Phi(\xi)=|\xi|^p+\Phi_1(\xi),
	\end{equation}
	where $p$ is a positive real number and $|\Phi_1(\xi)|\lesssim 1+|\xi|^q$, with $0\leq q<p$.  They proved that the Cauchy problem for \eqref{XM01} is locally well-posed in $H^{s}(\R)$, $s>-p/2$, with $p>3$. Also, they showed that for $p\geq 2$, there does not exist any $T>0$ such that the data-solution map $v_0\in H^s(\R)\mapsto v\in C([0,\,t]:H^s(\R))$ is  $C^2$- differentiable at the origin (see \cite{CM}).  When the nonlinearity in \eqref{XM01}  is $(v^{k+1})_x$, $k>1$ (generalized KdV nonlinearity), they obtain some  local well-posedness results  for the data with Sobolev regularity below $L^2(\R)$, see \cite{CM1}. Also, an n-dimensional dissipative version of the KdV equation \eqref{XM01} was considered in Carvajal, Esfahani and Panthee \cite{CMA}, where they prove well-posedness and ill-posedness results in anisotropic Sobolev spaces, they also study the dissipative limit of the solution when $\eta$ goes to zero. Finally, in \cite{CE} Carvajal and Esquivel proved local well-posedness for \eqref{g-KdV-B} in $H^s(\R)$ for $s>-p/2$, when $2\leq p< 3$, improving the result in \cite{MR1}. In the next figure we have a  resume of the former results.\\
	\begin{figure}[h]
		\centering
	\includegraphics[scale=0.25]{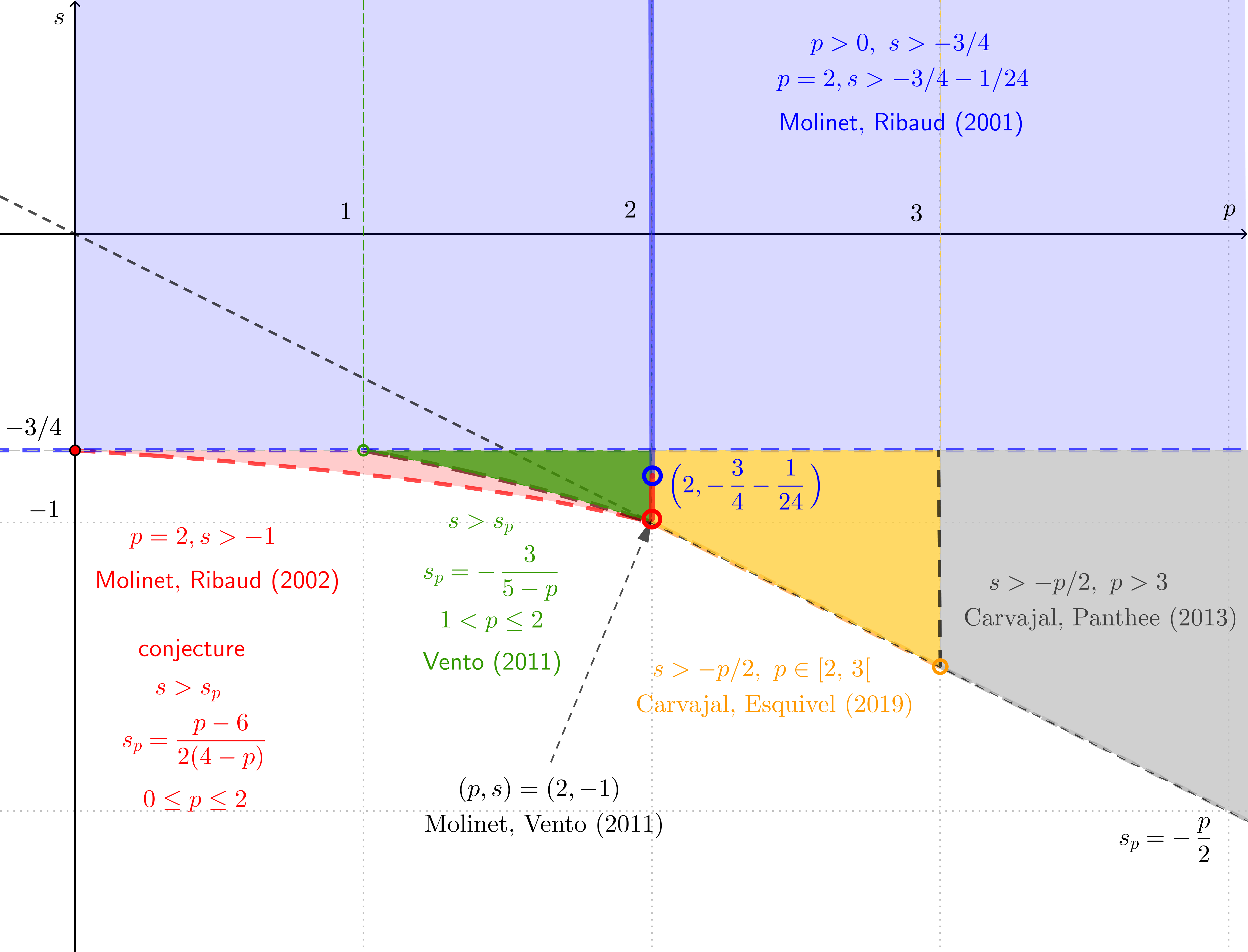}
	\end{figure}

	In our work, we use the framework developed in \cite{MV} to establish the following results: 
		
	\begin{theorem}\label{teo1}
		Let $p\geq 2$.
		The Cauchy problem associated to \eqref{g-KdV-B} is locally analytically well-posed in $H^{-p/2}(\mathbb{R})$. Moreover, at every point $u_0 \in H^{-p/2}(\R)$ there exist $T=T(u_0)>0$ and $R=R(u_0)>0$ such that the solution-map $u_0 \to u$ is analytic from the ball centered at $u_0$ with radius $R$ of $H^{-p/2}(\R)$ into $C([0,T]; H^{-p/2}(\R) )$. Also, the solution $u$ belongs to $C((0,\infty); H^{\infty}(\R) )$.
	\end{theorem}
	
	\begin{theorem}\label{ill}
	Let $p\geq 2$.
		The Cauchy problem associated to \eqref{g-KdV-B} is ill- posed in $H^{s}(\R)$ for $s<-p/2$: there exist $T>0$  such that for any $t\in [0,T]$ the flow-map $u_0 \to u(t)$ constructed in Theorem \ref{teo1} is discontinuous at the origin from  $H^{s}(\R)$ to $H^{s}(\R)$. 
	\end{theorem}

\begin{remark}\label{c2ill} i) Following  the same steps in the proof of Theorem. \ref{ill} we also obtain that the flow-map is not $C^2$ at the origin for $s< 3/2-p/4$, when $0\leq p\leq 2$. This allow us to conclude that we can not use contraction method to prove local well-posedness in this regularity.

ii) We observe that if $u$ is a solution of the \eqref{g-KdV-B} then $\|u\|_{L^2_x} \leq \|u_0\|_{L^2_x}$. Also, the local well-posedness result in \cite{CM} and \cite{CE} gives a global well-posedness result in $L^2(\mathbb{R})$ with $u \in C((0, \infty); H^{\infty}(\mathbb{R}))$ (see \cite{CM} and \cite{CE}) and on  the other hand the inequality \eqref{elem02} and the Section \ref{We-Pos} imply that, for any $0 < t < T$ there exists $t' \in ]0,t[$, such that $u(t') \in L^2(\mathbb{R})$. Thus the solution in Theorem \ref{teo1} also belongs to $C((0, \infty); H^{\infty}(\mathbb{R}))$.

	\end{remark}

These results extends the previous results in \cite{CM} (with $\Phi_1\equiv 0$ and $\eta=1$)  and in \cite{CE} for $s=-p/2$ and $p\ge 2$.

The plan of this paper is as follows. In Section 2 we fix some notations, define the spaces when we perform the iteration process,  prove some useful inequalities and recall some important results. In Section 3 we establish linear estimates related to the Duhamel operator, associated to the \eqref{g-KdV-B} equation. In Section 4, we prove the crucial result in this work: the bilinear estimates. In Section 5 we prove the Theorem \ref{teo1} and finally, in the Section 6 we prove the ill-posedness results.

	%%%%%%%%%%%%%%%%%%%%%%%%%%%%%%%%%%%%%%%%%%%%%%%%%%%%%%%%%%%%%%%%%%%%%%%%
	%%%%%%%%%%%%%%%%%%%%%
	
	\section{Notations and preliminaries results}

For $A,B>0$, we write $A\lesssim B$ when there exists $c>0$ such that $A\leq cB$. When the constant $c$ is small we write $A\ll B$. We write $A\sim B$ to denote that $A\lesssim B\lesssim A$. Also, we may  write $A\lesssim_{\alpha}B$, to express that the constant $c$ depends on $\alpha$. Given $u=u(t,x)\in \mathcal{S}'(\R^2)$,  we denote by  $\mathcal{F}u$ (or $\tilde{u}$), $\mathcal{F}_xu$ (or $\hat{u}$) and $\mathcal{F}_tu$ its Fourier transform in space-time, space and time respectively. Analogously, for the inverse Fourier transform we write $\mathcal{F}^{-1}u$, $\mathcal{F}^{-1}_\xi u$ and $\mathcal{F}^{-1}_{\tau}u$.

We work with the usual Lebesgue spaces as $L^p_x(\R)$, $L^p_t(\R)$ and $L^p_xL^q_t$. By simplicity we write $L^p_xL^p_t$  as $L^p$. The non-homogeneous Sobolev spaces are endowed with the norm $\|f\|_{H^s}=\|\cjap{\cdot}^s \hat{f}\|_{L^2}$, where $\cjap{x}=(1+|x|^2)^{1/2} \sim 1+|x|$ is the japanese bracket.

In order to define our functional spaces, we recall the Littlewood-Paley multipliers. Let us fix $\eta\in C^{\infty}_0(\R)$, such that $\eta\geq 0$, $\supp\eta\subset [-2,\,2]$ and $\eta\equiv 1$ on $[-1,\,1]$. A dyadic number is any number $N$ of the form $2^{j}$, where $j\in \mathbb{Z}$. With this notation, any sum over the dummy variable $M$, $N$ or $L$ is understood to be over dyadic numbers unless otherwise specified. Define $\varphi(\xi)=\eta(\xi)-\eta(2\xi)$ and $\psi(\tau,\,\xi)=\varphi(\tau-\xi^3)$. Using the notation $f_N(y)=f(y/N)$, we define, for $u\in \mathcal{S}'(\R^2)$ the Fourier multipliers
\begin{equation*}
\FT_x\{ \PN u(t,\cdot) \}(\xi)=\varphi_N(\xi)\hat{u}(t,\xi) \ \ \ \ \textrm{and}\ \ \ \ \FT\{\QL u\}(\tau,\,\xi)=\psi_L(\tau,\,\xi)\tilde{u}(\tau,\,\xi).
\end{equation*}

Because, rougly speaking, $\PN$ localizes in the annulus $\{|\xi|\sim N\}$ and $\QL$ localizes in the region $\{|\tau-\xi^3|\sim L\}$, they are so-called the Littlewod-Paley projections. We can define more projections like
\begin{equation*}
P_{\lesssim N}u = \sum_{\smath{M\lesssim N}}P_M u \ \ \textrm{or}\  \ \  Q_{\ll L}u= \sum_{\smath{M\ll L}}Q_M u,
\end{equation*}
and etc.

Associated to the equation \eqref{g-KdV-B}, we have the following integral equation
\begin{equation}\label{pre01}
u(t)=S_p(t)u_0-\frac{1}{2}\int_0^t S_{p}(t-t')\partial_xu^2(t')dt',\ \ \ t\geq 0,
\end{equation}
where the linear semi-group $S_p(t)=e^{-t(\partial_x^3+L_p)}=e^{-t\partial_x^3}e^{-tL_p}$, associated to  \eqref{g-KdV-B}, is given by 
\begin{equation}\label{pre02}
\FT_x\{S_{p}(t)f\}(\xi)=e^{it\xi^3-t|\xi|^p}\hat{f}(\xi), \ \ \ t\geq 0.
\end{equation}
We observe that $e^{-t\partial_x^3}$ is the unitary group associated to KdV equation and also, $e^{-tL_p}$, given by $e^{-tL_p}f=\mathcal{F}_{\xi}^{-1}\{e^{-t|\,\cdot\,|^p}\hat{f}(\,\cdot\,)\}$ is the semi-group associated to $\partial_tu+L_pu=0$.  We define the two-parameter linear operator $\emph{W}_{p}(t,\,t')=e^{-t\partial_x^3-|t'|L_p}$, given by
\begin{equation}\label{pre03}
\FT_x \{W_{\!p}(t,\,t')f\}(\xi)=e^{it\xi^3-|t'||\xi|^p}\hat{f}(\xi),\ \ \ t,t'\in \R.
\end{equation}
If $t=t'$, $t\in\R\mapsto W_{\!p}(t,t)$  is clearly an extension to $\R$ of $S_{p}(t)$.
Instead of use the integral equation \eqref{pre01}, we will apply a fixed-point argument to the following extension
\begin{equation}\label{pre04}
u(t)=\eta(t)W_{\!p}(t,t)u_0-\smath{\frac{1}{2}}\eta(t)\chi_{\!\,_{\R_{+}}}\!\!(t)\!\int_0^t\! W_{\!p}(t-t',t-t')\partial_x u^2(t')dt'-\smath{\frac{1}{2}}\eta(t)\chi_{\!\,_{\R_{-}}}\!\!(t)\!\int_0^t \!W_{\!p}(t-t',t+t')\partial_xu^2 (t')dt'\!,   
\end{equation}
$t\in\R$. Of course, if $u$ solves \eqref{pre04}, then $u|_{[0,\,T]}$ solves \eqref{pre01} in $[0,\,T]$, $T<1$.

The iteration process will be applied in the  Besov version of classical Bourgain Spaces, which we will be defined now, following \cite{MV}. For $s,b\in \R$, the space $X_{p}\pow{s,\,b,\,q}$ ($q=1$) is the weak closure of the test functions that are uniformly bounded by the norm
\begin{equation}\label{xspace}
\norma{u}{X^{s,b,q}_p}=\left(   \sum_{N}\left[  \sum_{L} \cjap{N}^{sq}\cjap{L+N^p}^{bq}\norma{\PN \QL u}{L^2}^q  \right]^{2/q}  \right)^{1/2}.
\end{equation}
In order to control the \textit{high-high} interaction in the nonlinearity, we introduce\footnote{The authors in \cite{MV}  were inspired by \cite{tao}. } for $b=\pm 1/2$, the space $Y\pow{s,\,b}$ endowed with the norm
\begin{equation}\label{yspace1}
\norma{u}{Y^{s,b}_p}=\left(   \sum_{N}\left[   \cjap{N}^s\norma{\IFT\{(i(\tau-\xi^3)+|\xi|^p+1)^{b+1/2}\phiN\tilde{u}\}} {L^1_tL^2_x}\right]^{2}  \right)^{1/2},
\end{equation}
such that
\begin{equation}\label{yspace2}
\norma{u}{Y^{-p/2,1/2}_p}=\left(   \sum_{N}\left[   \cjap{N}^{-p/2}\norma{(\partial_t+\partial_x^3+L_p +I)\PN u}{L^1_tL^2_x} \right]^{2}  \right)^{1/2},
\end{equation}
Thus, we form the resolution space $\mathcal{S}\pow{s}=X\pow{s,\,1/2,\,1}_p+Y\pow{s,\,1/2}_p$ and the nonlinear space $\mathcal{N}\pow{s}=X\pow{s,\,-1/2,\,1}_p+Y\pow{s,\,-1/2}_p$, endowed with the usual norm:
\begin{equation*}
\|u\|_{X+Y}=\inf\bigl\{ \|u_1\|_{X}+\|u_2\|_Y: u=u_1+u_2, \ \textrm{with}\ \ u_1\in X,\, u_2\in Y   \bigr\}.
\end{equation*}
From now on we work with the resolution space $\mathcal{S}\pow{-p/2}$ and the nonlinear space $\mathcal{N}\pow{-p/2}$. Remembering that $e^{-t\partial_x^3}f=\IFT_{\xi}\{ e^{it(\cdot)^3}\hat{f}(\cdot)\}$ is the group associated to the KdV equation, we have the following result:
\begin{lemma}\label{lema1} For any $\phi\in L^2_x(\R)$, we have
	\begin{equation*}
	\left(  \sum_{L}\bigl[ L^{1/2}\|\QL(e^{-t\partial_x^3}\phi)\|_{L^2} \bigr]^2 \right)^{1/2}\lesssim\, \|\phi\|_{L^2_x}.
	\end{equation*}
\end{lemma}
\begin{proof} See \cite{MV}.
	
\end{proof}

\begin{lemma} \label{lema2}
	\begin{enumerate}%[(1)]
		\item 	For each dyadic $N$, we have
		\begin{equation}\label{elem01}
		\|(\partial_t+\partial_x^3)\PN u\|_{\!_{L^1_tL^2_x}}\lesssim \,\| \PN u\|_{\!\,Y_p\pow {0,1/2}}.
		\end{equation}
		\item For all $u\in \mathcal{S}\pow{-p/2}$, with $p> 0$, 
		\begin{equation}\label{elem02}
		\|u\|_{\!_{L^2}}\lesssim\, \|u\|_{\!\,\mathcal{S}\pow {-p/2}}.
		\end{equation}
		\item For all $u\in \mathcal{S}^0$, 
		\begin{equation}\label{elem03}
		\left(  \sum_{L}\bigl[L^{1/2}\|\QL u\|_{\!_{L^2}} \bigr]^2 \right)^{1/2}\lesssim \ \  \|u\|_{\!\,\mathcal{S}\pow {0}}.
		\end{equation}
	\end{enumerate}

\end{lemma}

\begin{proof} We will prove only \eqref{elem02}, the proofs of \eqref{elem01} and \eqref{elem03} practically are given in \cite{MV}. As 
	$$
	\|u\|_{L^2} \sim \left( \sum_N  \|P_N u\|_{L^2}^2 \right)^{1/2} \quad \textrm{and} \quad 
	\|u\|_{\mathcal{S}\pow{-p/2}} \sim \left( \sum_N  \|P_N u\|_{\mathcal{S}\pow{-p/2}}^2 \right)^{1/2},
	$$
	it is sufficient to prove that
	\begin{equation}\label{eq19}
	\|P_N u\|_{L^2} \lesssim   \|P_N u\|_{\mathcal{S}\pow{-p/2}}.
	\end{equation}
	Remembering the definition of our resolution space, it suffices to prove \eqref{eq19} with $ \|P_N u\|_{X\pow{-p/2,1/2,1}}$ and  with $  \|P_N u\|_{Y^{-p/2,1/2}}$  in the right-hand side.
	%The resolution space  $\mathcal{S}\pow{-p/2}$ is given by %$\mathcal{S}\pow{-p/2}=X\pow{-p/2,1/2,1}+Y\pow{-p/2,1/2}$.
	%
	For the first, noting that $\langle L+N^p \rangle^{1/2} \gtrsim \langle N \rangle^{p/2}$, we have
		\begin{align*}\label{eq20}
		\|P_N u\|_{X\pow{-p/2,1/2,1}}&\sim\sum_{L} \langle N \rangle^{-p/2}\langle L+N^p \rangle^{1/2} \|P_N  Q_L u\|_{L^2} \\
		&\gtrsim \sum_{L} \|P_N  Q_L u\|_{L^2} \geq  \left( \sum_{L} \|P_N  Q_L u\|_{L^2}^2 \right)^{1/2}\\
		& \sim   \|P_N u\|_{L^2}.
		\end{align*}	
	%
	% 
	%First we will to prove that $ \|P_N u\|_{L^2} \lesssim  \|P_N u\|_{X\pow{-p/2,1/2,1}}$, we have
	%
	%$$
	%\|P_N u\|_{X\pow{-p/2,1/2,1}} \sim \sum_{L} \langle N \rangle^{-p/2}\langle L+N^p \rangle^{1/2} \|P_N  Q_L u\|_{L^2}
	%$$
	%and $\langle L+N^p \rangle^{1/2} \gtrsim \langle N \rangle^{p/2}$, we thus infer that
	%\begin{equation}\label{eq20}
	%  \|P_N u\|_{X^{-p/2,1/2,1}} \gtrsim  \sum_{L} \|P_N  Q_L u\|_{L^2} \geq  \left( \sum_{L} \|P_N  Q_L u\|_{L^2}^2 \right)^{1/2} \sim   \|P_N u\|_{L^2}.
	%\end{equation}
	%
	%
	%Finally  we will to prove that $
	%\|P_N u\|_{L^2} \lesssim  \|P_N u\|_{Y^{-p/2,1/2}}$.  Since
	%$$
	%\|P_N u\|_{L^2}= \|\varphi_N \widehat{u}\|_{L^2}=\|\mathcal{F}_t^{-1}  \left( %\varphi_N \tilde{u}  \right) \|_{L^2},
	%$$ 
	%then
	For the second inequality, since $\|P_N u\|_{L^2}= \|\varphi_N \widehat{u}\|_{L^2}=\|\mathcal{F}_{\tau}^{-1}  \left( \varphi_N \tilde{u}  \right) \|_{L^2}$, then
	\begin{equation}\label{eq21}
	\begin{split}
	\| P_N u \|_{L^2} =& \|\mathcal{F}_{\tau}^{-1}  \left( \dfrac{\varphi_N}{i(\tau -\xi^3) +|\xi|^p+1} \right) *{\!_{\!_t}}\, \mathcal{F}_{\tau}^{-1}  \left(  \left\{  i(\tau -\xi^3) +|\xi|^p+1 \right\}  \varphi_N  \tilde{u}  \right)\|_{L^2}\\
	\leq & \|\,   \|\mathcal{F}_{\tau}^{-1}  \left( \dfrac{\varphi_N}{i(\tau -\xi^3) +|\xi|^p+1} \right)\|_{L^2_t} \, \| \mathcal{F}_{\tau}^{-1}  \left(  \left\{  i(\tau -\xi^3) +|\xi|^p+1 \right\}  \varphi_N  \tilde{u}  \right)\|_{L^1_t}\|_{L_{\xi}^2}\\
	\leq &    \|\mathcal{F}_{\tau}^{-1}  \left( \dfrac{\varphi_N}{i(\tau -\xi^3) +|\xi|^p+1} \right)\|_{L_{\xi}^{\infty}L^2_t} \, \| \mathcal{F}_{\tau}^{-1}  \left(  \left\{  i(\tau -\xi^3) +|\xi|^p+1 \right\}  \varphi_N  \tilde{u}  \right)\|_{L_{\xi}^2 L^1_t}\\
	\leq &    \|\mathcal{F}_{\tau}^{-1}  \left( \dfrac{\varphi_N}{i(\tau -\xi^3) +|\xi|^p+1} \right)\|_{L^2_t L_{\xi}^{\infty}} \, \| \mathcal{F}_{\tau}^{-1}  \left(  \left\{  i(\tau -\xi^3) +|\xi|^p+1 \right\}  \varphi_N  \tilde{u}  \right)\|_{ L^1_t L_{\xi}^2}.
	\end{split}
	\end{equation}
	Using the definition of $\| \cdot \|_{Y^{0,1/2}}$ we get
	%Using the definition of $\| \cdot \|_{Y^{0,1/2}}$ we conclude that
	\begin{equation}\label{eq22}
	\| P_N u \|_{L^2} \lesssim     \|\mathcal{F}_\tau^{-1}  \left( \dfrac{\varphi_N}{i(\tau -\xi^3) +|\xi|^p+1} \right)\|_{L^2_t L_{\xi}^{\infty}} \, \|P_N u \|_{Y^{0,1/2}}.
	\end{equation}
	To estimate the norm of  inverse Fourier transform above, we note that
	%and 
	\begin{equation}\label{1eq21}
	\begin{split}
	\left|\mathcal{F}_{\tau}^{-1}  \left( \dfrac{\varphi_N}{i(\tau -\xi^3) +|\xi|^p+1} \right)(t)\right|=& |\varphi_N |\, \left| \int_{\R} e^{2\pi i t \tau } \dfrac{1}{i(\tau -\xi^3) +|\xi|^p+1} d\tau \right|\\
	=& |\varphi_N |\, \left| \int_{\R} e^{2\pi i t x } \dfrac{1}{i x +|\xi|^p+1} dx \right|\\
	=& |\varphi_N |\, \left| \int_{\R} e^{2\pi i t x } \dfrac{1}{x -i(|\xi|^p+1)} dx \right|.
	\end{split}
	\end{equation}
	%It is not difficult to see that if $k \neq 0$ is  constant, then
	Now, if $k\neq0$ is a constant, we have that
	\begin{equation}\label{eq23a}
	\widehat{\left(\dfrac{x}{x^2+k^2}\right)}(t)=-\pi i\, \textrm{sgn} \, t e^{-2\pi k|t|},\quad  \widehat{\left(\dfrac{k}{x^2+k^2}\right)}(t)=\pi \,  e^{-2\pi k|t|}.
	\end{equation}
	For the first and the second identities we refers,  e.g.  \cite{Duoan} pg. 49 and  \cite{Stein} pg. 127, respectively. With these identities in hands, we obtain
	\begin{equation}\label{eq23}
	\begin{split}
	\mathcal{F}_x^{-1}{\left(\dfrac{1}{x-ik}\right)}(t)=& \mathcal{F}_x^{-1}\left(\dfrac{x}{x^2+k^2}\right)(t)+ i \mathcal{F}_x^{-1}\left(\dfrac{k}{x^2+k^2}\right)(t)\\
	=& \pi i \, (1+ \textrm{sgn} t) \, e^{-2\pi k|t|}\\
	=&\begin{cases}
	2 \pi i  \, e^{-2\pi k|t|},&\textrm{if} \quad t \geq 0,\\
	0,&\textrm{if} \quad t \leq 0.
	\end{cases}
	\end{split}
	\end{equation}
	Combining \eqref{1eq21} and \eqref{eq23} we get
	%consequently from \eqref{1eq21} and \eqref{eq23} we conclude 
	\begin{equation}\label{eq24}
	\begin{split}
	\left|\mathcal{F}_{\tau}^{-1}  \left( \dfrac{\varphi_N}{i(\tau -\xi^3) +|\xi|^p+1} \right)(t)\right|=&2 \pi  |\varphi_N |  \, e^{-2\pi (1+|\xi|^p) t} \chi_{\R^{+}}(t)\\
	\lesssim&  \, e^{-2\pi N^p t} \chi_{\R^{+}}(t).
	\end{split}
	\end{equation}
	Using \eqref{eq22} and \eqref{eq24}, we conclude that
	\begin{equation}\label{eq25}
	\begin{split}
	\| P_N u \|_{L^2} \lesssim &    N^{-p/2} \, \|P_N u \|_{Y^{0,1/2}}\lesssim  \|P_N u \|_{Y^{-p/2,1/2}}.
	\end{split}
	\end{equation}
\end{proof}

\begin{lemma}\label{lema3} \textbf{(Extension lemma)} Let $\mathcal{Z}$ be a Banach space of functions on $\R\times \R$ with the property that
	$$\|g(t)u(t,x)\|_{\!\mathcal{Z}} \lesssim \, \|g\|_{\!L^{\infty}_t}\|u(t,x)\|_{\!\mathcal{Z}}$$
	holds for any $u\in \mathcal{Z}$ and $g\in L^{\infty}_t(\R)$. Let $T$ be a spatial linear operator for which one has the estimate
	$$\|T(e^{-t\partial_x^3}\PN \phi)\|_{\!\mathcal{Z}}\lesssim\, \|\PN\phi\|_{\!L^2}$$
	for some dyadic $N$ and for all $\phi$. Then one has the embedding
	$$\|T(\PN u)\|_{\!\mathcal{Z}}\lesssim\, \|\PN u\|_{\!\mathcal{S}\pow 0}.$$
\end{lemma}
\begin{proof} See in \cite{MV} the comments before the Lemma 3.3 or see \cite{tao}. 
\end{proof}
As a consequence of this abstract result, using   the Kato smoothing effect  
	\begin{equation}\label{kato}
	\|\partial_xe^{-t\partial_x^3}\phi\|_{L^{\infty}_xL^2_t}\lesssim \, \|\phi\|_{L^2}, \ \ \ \ \ \forall \phi\in L^2,
	\end{equation}
	and that $e^{-t\partial_x^3}$ is a unitary operator in $L^2$, we obtain the following results
\begin{corollary}\label{cor} For any $u$, we have, for $p> 0$, that
	\begin{align}
	\|u\|_{\!_{L^{\infty}_tH\pow{-p/2}_x}} \lesssim {}& \|u\|_{\!\mathcal{S}\pow {-p/2}}, \label{elem04}\\
	\|\PN u\|_{\!_{L^{\infty}_tL^2_x}}\lesssim {}& N^{-1}\|\PN u\|_{\!\mathcal{S}^0} \label{elem05},
	\end{align}
	provided the right-hand side is finite.	
\end{corollary}

%%%%%%%%%%%%%%%%%%%%%%%%%%%FIM Notações%%%%%%%%%%%%%%%%%%%%%%%%%%%%%%%%%%	

\section{Linear Estimates}\label{es-line}

In this section we prove linear estimates related to operator $\emph{W}_p$ as well to the extension of the Duhamel operator introduced in \eqref{pre04}. We will do some adaptations of the arguments in \cite{MV}, in order to get the necessary estimates.  

\begin{proposition}\label{prop-a}
	For all $\phi\in H^{-p/2}(\R)$ and  $p>0$, we have
	\begin{equation}\label{esti-a}
	|| \eta(t)\emph{W}_{\!p}(t,\,t) \phi ||_{\mathcal{S}^{-p/2}}\lesssim || \phi ||_{H^{-p/2}}.
	\end{equation}
\end{proposition}

\begin{proof}
	Clearly, the left-hand side in \eqref{esti-a} is bounded  by   $\normaXp{\eta(t)\emph{W}(t,t) \phi}$. It suffices to show
		\begin{equation}\label{aag}
		\sum_{L}\langle{L+N^p}\rangle^{1/2} || P_{N}Q_{L}(\eta(t)\emph{W}_{p}(t,\,t) \phi) ||_{L^{2}}\lesssim ||P_{N}\phi ||_{L^{2}},
		\end{equation}
		After this, multiplying both sides by $\cjap{N}\pow{-p/2}$,  squaring and summing in $N$, we get the desired. In order to prove \eqref{aag}, first we note that	%
		\begin{align}
		|| P_{N}Q_{L}(\eta(t)\emph{W}_{p}(t,\,t) \phi) ||_{L_{xt}^{2}}& = || \phiN(\xi)\phiL(\tau - \xi^3)\mathcal{F}_{t}(e^{it\xi^3-|t| |\xi|^p}\eta(t))(\xi)\hat{\phi}(\xi) ||_{L_{\xi,\tau}^{2}}\nonumber\\
		&\lesssim|| \phiN(\xi)\phiL(\tau - \xi^3)\mathcal{F}_{t}(e^{-|t| |\xi|^p}\eta(t))(\tau-\xi^3) \phiN(\xi)\hat{\phi}(\xi)||_{L_{\xi,\tau}^{2}}\nonumber\\
		&\stackrel{(1)}{=}  || \phiL(\tau )\mathcal{F}_{t}(e^{-|t| |\xi|^p}\eta(t))(\tau) \phiN(\xi)\hat{\phi}(\xi)||_{L_{\xi,\tau}^{2}}\nonumber\\
		&\stackrel{(2)}{\leq}  ||\phiN(\xi)\phiL(\tau)\mathcal{F}_{t}(e^{-|t| |\xi|^p}\eta(t))(\tau) ||_{L_{\xi}^{\infty}L_{\tau}^{2}}|| \widehat{\PN\phi} ||_{L_{\xi,\tau}^{2}}\nonumber\\
		&= ||\phiN(\xi)P_{L}(e^{-|t| |\xi|^p}\eta(t)) ||_{L_{\xi}^{\infty}L_{t}^{2}}||\PN\phi ||_{L_{x,t}^{2}}, \label{aab}
		\end{align}
		where we using in (1)  the translation invariance of the $L^p$-norms and in $(2)$ the H\"older inequality.

	Adding in $L$  we obtain
		\begin{equation}\label{aaf}
		\begin{split}
		&\sum_{L}\langle{L+N^p}\rangle^{1/2}||| P_{N}Q_{L}(\eta(t)\emph{W}_{\!p}(t,\,t) \phi) ||_{L_{x,t}^{2}}\\
		&\lesssim||P_{N} \phi ||_{L_{x,t}^{2}}\sum_{L}\langle{L+N^p}\rangle^{1/2} ||\phiN(\xi)P_{L}(e^{-|t| |\xi|^p}\eta(t)) ||_{L_{\xi}^{\infty}L_{t}^{2}}.
		\end{split}
		\end{equation}
	To get the bound in \eqref{aag} we will prove that
	%For this we will show that
	\begin{equation}\label{aah}
	\sum_{L}\langle{L+N^{p}}\rangle^{1/2}||\phiN(\xi) P_{L}(e^{-|t| |\xi|^p}\eta(t)) ||_{L_{\xi}^{\infty}L_{t}^{2}} \lesssim 1.
	\end{equation}
	Spliting the summand into $L\leq \cjap{N}^p$ and $L\geq \cjap{N}^p$, the proof will be done in two cases. For the first case, applying Bernstein inequality in time, we have  
	
	\begin{equation}\label{aai}
	\begin{split}
	&\sum_{L\lesssim\, \langle{N}\rangle^{p}}\langle{L+N^p}\rangle^{1/2}||\phiN(\xi) P_{L}(e^{-|t| |\xi|^p}\eta(t)) ||_{L_{\xi}^{\infty}L_{t}^{2}}\,\lesssim\;\sum_{L\lesssim\, \langle{N}\rangle^{p}}\langle{N}\rangle^{p/2} L^{1/2} \sup_{|\xi|\sim N} ||e^{-|t| |\xi|^p}\eta(t) ||_{L_{t}^{1}}.
	\end{split}
	\end{equation}
	Noting that
	\[
	||e^{-|t| |\xi|^p}\eta(t) ||_{L_{t}^{1}} \lesssim\, \min\{ 1,\,|\xi|^{- p} \},
	\]
	then we have
	\begin{equation}\label{aaj}
	\sum_{L\lesssim\, \langle{N}\rangle^{p}}\langle{L+N^p}\rangle^{1/2}||\varphi_{N} P_{L}(e^{-|t| |\xi|^p}\eta(t)) ||_{L_{\xi}^{\infty}L_{\tau}^{2}} \lesssim \, \langle{N}\rangle^{p}  \min \{ 1,\,N^{- p} \}\lesssim 1.
	\end{equation}
For the second case, using  the following rearrangement $$\sum_{M}\sum_{N} a_{M,N}=\sum_{M}\sum_{N\lesssim M}a_{M,N}+\sum_{M}\sum_{N\gtrsim M}a_{M,N}=\sum_{M}\sum_{N\lesssim M}a_{M,N}+\sum_{M}\sum_{N\lesssim M}a_{N,M},$$
		one can see that
		\begin{equation}\label{aak}
		\begin{split}
		P_{L}(e^{-|t| |\xi|^p}\eta(t)) &=\,P_{L}\left[\sum_{M\gtrsim L}(\,P_{M}\eta(t)P_{\lesssim M}e^{- |t| |\xi|^{p}} + P_{\lesssim M}\eta(t)P_{M} e^{- |t| |\xi|^{p}}\,) \right]\\
		&=\, P_{L}(I) + P_{L}(II).
		\end{split}
		\end{equation}
	For the term $P_L(I)$, using H\"older inequality 
		\begin{align}
		\sum_{L\gtrsim \langle{N}\rangle^{p}}\langle{L+N^p}\rangle^{1/2}||\varphi_{N}P_{L}(I) ||_{L_{\xi}^{\infty}L_{\tau}^{2}}&\lesssim\sum_{L}L^{1/2}\sum_{M\gtrsim L}\|\phiN(\xi)P_M\eta(t)\|_{L^{\infty}_{\xi}L_t^2}\|\phiN(\xi)P_{\lesssim M}e^{-|t||\xi|^p}\|_{L^{\infty}_{x,t}}\nonumber\\ 
		&\lesssim \sum_{M}L^{1/2}\sum_{L\lesssim M}\|\phiN(\xi)P_M\eta(t)\|_{L^{\infty}_{\xi}L_t^2}\|\phiN(\xi)P_{\lesssim M}e^{-|t||\xi|^p}\|_{L^{\infty}_{x,t}} \nonumber\\
		&\lesssim \sum_{M} M^{1/2}\|\phiN(\xi)P_M\eta(t)\|_{L^{\infty}_{\xi}L_t^2}\|\phiN(\xi)P_{\lesssim M}e^{-|t||\xi|^p}\|_{L^{\infty}_{x,t}}.\label{aal}
		\end{align}
		But, because $\phiN(\xi)P_{\lesssim M}e^{-|t||\xi|^p}\lesssim\phiN(\xi)\arctan(M/\xi^p)$, then  the right-hand side of \eqref{aal} is bounded by
		\begin{align*}
		& \sum_{M} M^{1/2}\|P_M\eta\|_{L^2_t}\lesssim_{_{ \|\eta\|_{B\pow{1/2}_{\smath{2,1}} } } }1.
		\end{align*}
		Proceeding in a similar way for $P_{L}(II)$, we obtain
		\begin{equation}\label{aam}
		\sum_{L\gtrsim \langle{N}\rangle^{p}}\langle{L+N^p}\rangle^{1/2}||\varphi_{N}P_{L}(II) ||_{L_{\xi}^{\infty}L_{\tau}^{2}}\lesssim\,\sum_{M}M^{1/2}||\phiN P_{M}e^{- |t| N^p} ||_{L_{t}^{2}}\lesssim_{_{ \|e^{-|t|}\|_{\textrm{\.B}\pow{1/2}_{\smath{2,1}} } } }1,
		\end{equation}
		remembering that the homogeneous Besov space $\textrm{\.B}\pow{1/2}_{2,1}$ has a scaling invariance and  $e^{-|\cdot|} \in \textrm{\.B}\pow{1/2}_{2,1}$.
\end{proof}

\begin{lemma}\label{lem-a}
	Let $p>0$. For $\omega\in {\mathcal S}(\R^2)$, consider $\kappa_{p,\xi}$ defined on $\R$ by
	\[
	\kappa_{p,\xi}(t)\,=\,\eta(t)\varphi_{N}(\xi)\int_{\R}\dfrac{e^{it\tau}e^{(t-|t|)|\xi|^p} - e^{- |t| |\xi|^p}}{i\tau + |\xi|^p}\,\tilde{\omega}(\tau, \xi)\,d\tau.
	\]
	Then, for all $\xi\in\R$, it holds
	\begin{equation}\label{lla}
	\sum_{L}\langle{L+N^{p}}\rangle^{1/2}||P_{L}\kappa_{p,\xi} ||_{L_{t}^{2}}\lesssim \sum_{L}\langle{L+N^p}\rangle^{- 1/2}|| \varphi_{L}(\tau)\varphi_{N}(\xi)\,\tilde{\omega}||_{L_{\tau}^{2}}.
	\end{equation}
\end{lemma}
\begin{proof}
	As in \cite{MV}, adding and subtracting  $\eta(t)e^{(t-|t|)|\xi|^p}$ inside the integral (numerator), we can rewrite $k_{p,\xi}$ as
	\begin{equation}\label{llb}
	\begin{split}
	\kappa_{p,\xi}(t)& = \eta(t)e^{(t-|t|)|\xi|^p}\int_{|\tau|\leq 1}\dfrac{e^{it\tau}-1}{i\tau+|\xi|^p}\,\widetilde{\omega_{N}}d\tau + \eta(t)\int_{|\tau|\leq 1}\dfrac{e^{(t-|t|)|\xi|^p}-e^{-|t||\xi|^p}}{i\tau+|\xi|^p}\,\widetilde{\omega_{N}}d\tau\\[5pt]
	&\quad +  \eta(t)e^{(t-|t|)|\xi|^p}\int_{|\tau|\geq 1}\dfrac{e^{it\tau}}{i\tau+|\xi|^p}\,\widetilde{\omega_{N}}d\tau - \eta(t)\int_{|\tau|\geq 1}\dfrac{e^{-|t||\xi|^p}}{i\tau+|\xi|^p}\,\widetilde{\omega_{N}}d\tau,\\[5pt]
	& = (I) + (II) + (III) - (IV),  
	\end{split}
	\end{equation}
	where $\widetilde{\omega_{N}}$ is defined by $\widetilde{\omega_{N}}(\tau,\,\xi) = \varphi_{N}(\xi) \tilde{\omega}(\tau,\,\xi)$.
	By triangular inequality, it's suffices to prove the estimate \eqref{lla} with $P_L(I)$, $P_L(II)$, $P_L(III)$ and $P_L(IV)$ in place of $P_L\kappa_{p,\xi}$. 
	\ \\
	
	\noindent
	{\bf Term\, (IV).} With \eqref{aah} in mind and  performing a straighfoward calculations we get
		\begin{align}
		\|P_L(IV)\|_{L^2_t}&\lesssim \|\phiN(\xi) P_L(\eta(t)e^{-|t||\xi|^p})\|_{L^2_t}\int_{|\tau|\geq 1}\cjap{i\tau+|\xi|^p}^{-1}\widetilde{\omega_{N}}d\tau.\nonumber\\
		&\lesssim \int_{|\tau|\geq 1}\cjap{i\tau+|\xi|^p}^{-1}\widetilde{\omega_{N}}d\tau .\label{4llc}
		\end{align}
		Also, because $\cjap{i\tau+|\xi|^p}\gtrsim \cjap{1+|\xi|^p}$ then we have
		\begin{align}
		\int_{|\tau|\geq 1}\cjap{i\tau+|\xi|^p}^{-1}\widetilde{\omega_{N}}d\tau&\lesssim \sum_{L\geq 1}\|\cjap{1+N^p}^{-1}\phiL\widetilde{\omega_{N}}\|_{L^1_{\tau}}\nonumber\\
		&\lesssim \sum_{L}\cjap{L+N^p}^{-1}L^{1/2}\|\phiL \widetilde{\omega_{N}}\|_{L^2_{t}},\label{4lld}
		\end{align}
		where in the last line we use the Cauchy-Schwarz in $\tau$. This yields the desired bound. 
	
	\ \\

	\noindent
	{\bf Term (II).} Taking account that
		\begin{align}\label{2lle}
		\int_{|\tau|\leq 1}\dfrac{|\widetilde{\omega}_{N}(\tau)|}{|i\tau+|\xi|^p|}\,d\tau&\leq\left(\int_{|\tau|\leq 1}\dfrac{|\phiN(\xi)|\langle{i\tau+|\xi|^p}\rangle}{|i\tau+|\xi|^p|^2}\,d\tau\right)^{1/2}\left(\int_{|\tau|\leq 1}\dfrac{|\widetilde{\omega}_{N}(\tau)|^2}{\langle{i\tau+|\xi|^p}\rangle}\,d\tau\right)^{1/2}\nonumber\\
		&\lesssim \frac{\cjap{N}^{p/2}}{N^{p}}\sum_{L}\cjap{L+N^p}^{-1/2}\|\phiL \widetilde{\omega}_{N}(\tau)\|_{L^2_{\tau}}
		\end{align}
		thus
	\begin{equation}\label{2llf}
		\begin{split}
		\sum_{L}\langle{L+N^p}\rangle^{1/2}||P_{L}(II) ||_{L_{t}^{2}}&\lesssim\,\langle{N}\rangle^{p/2}N^{-p} \sum_{L}\langle{L+N^p}\rangle^{1/2}||P_{L}(\phiN(\xi)\eta(t)(e^{(t-|t|)|\xi|^p}-e^{- |t||\xi|^{p}})) ||_{L_{t}^{2}}.\\[2pt]
		& \qquad \times\sum_{L}\langle{L+N^{p}}\rangle^{- 1/2}||\varphi_{L}\widetilde{\omega}_{N}||_{L_{\tau}^{2}}.
		\end{split}
		\end{equation}
		We need to prove that
		\begin{equation}\label{2llh}
		\sum_{L}\langle{L+N^p}\rangle^{1/2}||P_{L}(II) ||_{L_{t} ^2   }\,\lesssim\,\sum_{L}\langle{L+N^{p}}\rangle^{- 1/2}||\varphi_{L}\widetilde{\omega}_{N}||_{L_{\tau}^{2}},\ \ \ \ \ \ \forall\,N \, \textrm{dyadic}.\ 
		\end{equation}
		In view of \eqref{2llf} it suffices to prove that
		\begin{equation}\label{2llf1}
		\sum_{L}\langle{L+N^p}\rangle^{1/2}||P_{L}(\phiN(\xi)\eta(t)(e^{(t-|t|)|\xi|^p}-e^{- |t||\xi|^{p}})) ||_{L_{t}^{2}}\lesssim \cjap{N}^{p/2}.
		\end{equation}
		For technical reasons, we will divide the proof in two cases, namely,  $N\geq 1$ and  $N<1$. For the first case, by triangular inequality, we have
		\begin{align}
		\sum_{L}\langle{L+N^p}\rangle^{1/2}||P_{L}(\phiN(\xi)\eta(t)(e^{(t-|t|)|\xi|^p}-e^{- |t||\xi|^{p}})) ||_{L_{t}^{2}}\hspace{6cm} \nonumber\\
		\lesssim \underbrace{\sum_{L}\langle{L+N^{p}}\rangle^{1/2}||P_{L}(\phiN(\xi)\eta(t)e^{(t-|t|)|\xi|^{p}}) ||_{L_{t}^{2}}}_{{\bf J}}+\underbrace{\sum_{L}\langle{L+N^{p}}\rangle^{1/2}||P_{L}(\phiN(\xi)\eta(t)e^{-|t||\xi|^{p}}) ||_{L_{t}^{2}}}_{{\bf K}}.\label{2llf2}
		\end{align}
		One can see that ${\bf K}\,\lesssim\,1$, thanks to estimate \eqref{aah}.   We will estimate the first term. Denoting
		$\theta_{p}(t) = \eta(t)e^{(t-|t|)|\xi|^p}$,  the estimates
		\begin{equation}\label{decay}
		|\hat{\theta}_{p}(\tau)|\,\lesssim|\tau|^{-1}\ \ \ \ \textrm{and}\ \ \ \ |\hat{\theta}_{p}(\tau)|\,\lesssim \langle{\xi}\rangle^{p}|\tau|^{-2},
		\end{equation}
		yields from one and two integrations by parts, respectively.
	Now,  splitting the summand in a convenient  way and use the estimates in \eqref{decay} we get
		\begin{equation}\label{2llg}
		\begin{split}
		{\bf J} & =\,\sum_{L\leq 1}\langle{L+N^p}\rangle^{1/2}||\phiL(\tau)\phiN(\xi)\hat{\theta}_{p}(\tau) ||_{L_{\tau}^{2}}+ \sum_{1\leq L\leq\langle{N}\rangle^p}\langle{L+N^p}\rangle^{1/2}||\phiL(\tau)\phiN(\xi)\hat{\theta}_{p}(\tau) ||_{L_{\tau}^{2}}\\[2pt]
		&\qquad + \sum_{L\geq\langle{N}\rangle^p}\langle{L+N^p}\rangle^{1/2}||\phiL(\tau)\phiN(\xi)\hat{\theta}_{p}(\tau) ||_{L_{\tau}^{2}}\\[2pt]
		&\lesssim\,\sum_{L\leq 1}\langle{N}\rangle^{p/2}L^{1/2}||\theta_{p}||_{L_{t}^{1}} +\sum_{1\leq L\leq\langle{N}\rangle^p}\langle{N}\rangle^{p/2}L^{- 1}L^{-1/2}||\theta_{p}||_{L_{t}^{1}}+ \sum_{ L\geq\langle{N}\rangle^p}\,\langle{L}\rangle^{1/2}L^{-2}L^{1/2}||\theta_{p}||_{L_{t}^{1}}\,\langle{N}\rangle^{p}\\[2pt]
		&\lesssim \,\langle{N}\rangle^{p/2}.
		\end{split}
		\end{equation}
		Therefore, remembering that $\textbf{K}\lesssim1\leq \cjap{N}\pow{p/2}$ and  combining this fact, the estimates \eqref{2llg} and \eqref{2llf2} with \eqref{2llf} we conclude the estimate \eqref{2llh} for $N\geq 1$. 
	%
	%
	%
	%%%
	%
	%
	%
	%
	The case  $ N \leq1$ will be treated in a different way: we will use a Taylor expansion. The identity
		\[
		(e^{(t-|t|)|\xi|^{p }  }-1) - (e^{-|t||\xi|^{p}}-1) = \sum_{n=1}^{\infty}\dfrac{\chi_{\R_{-}}\!(t)(2t)^{n}}{n!}\,|\xi|^{pn} - \sum_{n=1}^{\infty}\dfrac{(-|t|)^{n}}{n!}\,|\xi|^{pn},
		\]
		allows us to conclude that
		\begin{align}
		\sum_{L}\langle{L+N^p}\rangle^{1/2}||P_{L}(\phiN(\xi)\eta(t)(e^{(t-|t|)|\xi|^p}-e^{- |t||\xi|^{p}}) ) ||_{L_{t}^{2}}\hspace{3cm}\nonumber\\
		\lesssim  \sum_{n=1}^{\infty}\frac{|\xi|^{pn}}{n!}\sum_{L}\langle{L}\rangle^{1/2}\left[ \|P_L(|t|^n\eta(t))\|_{L^2_t}+2^n\|P_L(t^n\eta(t)\chi_{\R_{-}}(t))\|_{L^2_t}  \right]\nonumber\\
		\lesssim N^p\sum_{n=1}^{\infty}\dfrac{1}{n!}\left[ ||\,|t|^n\eta(t) ||_{\emph{B}_{2,1}^{1/2}} + 2^n||t^n\eta(t)\chi_{\R_{-}}(t) ||_{\emph{B}_{2,1}^{1/2}}\right].\hspace{2cm}\label{2lli}
		\end{align}
		Because $\emph{H}^{1}\hookrightarrow B_{2,1}^{1/2}$ and $||\chi_{\R_{-}} f ||_{\emph{H}_{t}^{1}}\,\lesssim\,||f ||_{\emph{H}_{t}^{1}}$ if $f(0) = 0$, the right-hand side of \eqref{2lli} is
	\begin{equation}\label{2llj}
	\begin{split}
	\lesssim& \;N^p\sum_{n=1}^{\infty} \dfrac{1}{n!}\,2^n ||\,|t|^n\eta(t) ||_{\emph{H}_{t}^{1}}\,\lesssim\, N^p.
	\end{split}
	\end{equation}
Combining this last estimate with \eqref{2llf} we conclude the estimate \eqref{2llh} for $N<1$. This finishes the estimate of {\bf Term (II).}

	\noindent
	{\bf Term (I).} \, Using a Taylor expansion of $e^{it\tau}$, the Cauchy-Schwarz inequality in $\tau$ and remembering the estimates of the integrals in \eqref{2lle}, we obtain
		\begin{align}
		||P_{L}(I) ||_{L_{t}^{2}}\,&\lesssim\,\sum_{n=1}^{\infty}\dfrac{1}{n!}||\phiN(\xi)P_{L}(t^n\theta_p) ||_{L_{t}^{2}} \int_{|\tau|\leq 1}\,\dfrac{|\tau|^n}{| i\tau+|\xi|^p|}\,|\widetilde{\omega}_{N}(\tau)|\,d\tau\nonumber\\
		&\lesssim \sum_{n=1}^{\infty}\dfrac{1}{n!}||P_{L}(\phiN(\xi)t^n\theta_p) ||_{L_{t}^{2}}\left[ \int_{|\tau|\leq 1}\,\dfrac{|\widetilde{\omega}_{N}(\tau)|^2}{\langle{ i\tau+|\xi|^p}\rangle}\,d\tau \right]^{1/2}\left[ \int_{|\tau|\leq 1}\,\dfrac{|\phiN(\xi)||\tau|^2\langle{i\tau+|\xi|^p}\rangle}{| i\tau+|\xi|^p|^2}\,d\tau\right]^{1/2}.\nonumber\\
		&\lesssim \cjap{N}^{p/2}N^{-p}\sum_{n=1}^{\infty}\dfrac{1}{n!}||\phiN(\xi)P_{L}(t^n\theta_p) ||_{L_{t}^{2}} \sum_{L_1}\cjap{L_1+N^p}^{-1/2} \|\varphi_{L_1}\widetilde{\omega}_{N}\|_{L^2_{\tau}}.\label{1lla}
		\end{align}
		Thus,  it suffices to show that 
		\begin{equation}\label{1lla2}
		\sum_{L}\cjap{L+N^p}^{1/2}\sum_{n=1}^{\infty}\frac{1}{n!}\|\phiN(\xi)P_L(t^n\theta_p(t))\|_{L^2_t}\lesssim \cjap{N}^{p/2}.
		\end{equation}

 Again, using \eqref{decay}, we get  
		\begin{equation}\label{decay2}
		|\,\widehat{t^n\theta_{p}(t)}(\tau)\,|\, \leq\, 2^n \|\theta_{p}\|_{L^1_t}\leq 2^n\min\{ |\tau|^{-1},\,\cjap{\xi}^{p}|\tau|^{-2} \}.
		\end{equation}
		With this in hands and arguing as in \eqref{2llg}, we have that the left-hand side of \eqref{1lla2} is
		$$\lesssim \cjap{N}^{p/2}\sum_{n=1}^{+\infty}\frac{2^n}{n!}\lesssim\cjap{N}^{p/2},$$
		the desired bound.\\

	\noindent
	{\bf Term (III).} \,Writting $ \hat{g}(\tau) := \dfrac{\widetilde{\omega}_{N}(\tau)}{it+|\xi|^{p}}\,\chi_{\{|\tau| \geq 1\}}$, so we need to prove that
	\begin{equation}\label{3lla}
	\sum_{L}\langle{L+N^{p}}\rangle^{1/2}||P_{L}(\theta_p g) ||_{L_{t}^{2}}\,\lesssim\,\sum_{L}\langle{L+N^{p}}\rangle^{1/2}||P_{L} g ||_{L_{t}^{2}},
	\end{equation}
	noting that $||P_{L} g ||_{L_{t}^{2}}\lesssim \cjap{L+N^{p}}^{-1}\|\varphi_L\varphi_N\tilde{\omega}\|_{L^2_{\tau}}$. First, using a paraproduct decomposition as in \eqref{aak} we have
	\[
	P_{L}(\theta_pg) = P_{L}\left( \sum_{M\gtrsim L}( P_{\lesssim M}\theta_{p}P_{M} g +  P_{M}\theta_{p}P_{\lesssim M} g ) \right) = P_{L}(III_{1}) +  P_{L}(
	III_{2}).
	\]

		We estimate the contributions of these terms separately. In both cases, we divide the proof when $L \leq \langle{N}\rangle^{p}$ and $L > \langle{N}\rangle^{p}$.\\
	
	\noindent
	{\bf Term $(III_1)$}. For the sum over $ L\geq {\langle{N}\rangle}^{p}$,  rearranging the sums we have
		\begin{align}
		\sum_{L\geq {\langle{N}\rangle}^{p}}{\langle{L+N^{p}}\rangle}^{1/2}||P_{L}(III_1) ||_{L_{t}^{2}}\,&\lesssim\,\sum_{L\geq {\langle{N}\rangle}^{p}}\langle{L}\rangle^{1/2}\sum_{M\gtrsim L}\|P_{\lesssim M}\theta_p|_{L^{\infty}_t} ||P_{M} g ||_{L_{t}^{2}}\nonumber\\
		&\lesssim\,\sum_{M\geq {\langle{N}\rangle}^{p}}\sum_{L\leq  M}\langle{L}\rangle^{1/2} ||P_{M} g ||_{L_{t}^{2}}\nonumber\\
		\,&\lesssim\, \sum_{M}\langle{M}\rangle^{1/2}||P_{M} g ||_{L_{t}^{2}}.\label{3llb}
		\end{align}
		Now we deal with the sum over $ L\leq \langle{N}\rangle^{p}$.  Because $\supp(\widehat{P_{\lesssim M}\theta_p}) \subset \{ |\tau| \sim M\}\cup\{|\tau| \ll M \}\cap\supp(\widehat{\theta_p}) $, this case is divided into two subcases, namely, when $ supp(\widehat{\theta_p})\subset \{ |\tau| \sim M \} $ or when  $\supp(\widehat{\theta_p})\subset \{ |\tau| \ll M \}$.\\
		\indent
		For the first subcase,  applying the Bernstein inequality and rearranging the sums, we obtain
		\begin{align}
		\sum_{L\leq {\langle{N}\rangle}^{p}}{\langle{L+N^{p}}\rangle}^{1/2}||P_{L}(III_1) ||_{L_{t}^{2}}&\lesssim 
		\sum_{L\leq \cjap{N}^{p} } \cjap{ L+N^{p} }^{1/2}\sum_{M\gtrsim L}||P_{L}(P_{\lesssim M}\theta_{p} P_{M}g) ||_{L_{t}^{2}}\nonumber\\ &\lesssim\, \sum_{L\leq \cjap{N}^{p} }\cjap{N}^{p/2}\sum_{M\gtrsim L}||P_{L}(P_{M}\theta_{p} P_{M}g) ||_{L_{t}^{2}}\nonumber\\
		&\lesssim \sum_{L\leq \cjap{N}^{p} }\cjap{N}^{p/2}\sum_{M\gtrsim L} L^{1/2}||P_{M}\theta_p ||_{L_{t}^{2}}||P_{M} g ||_{L_{t}^{2}} \nonumber\\
		&\lesssim \sum_{M}\langle{N}\rangle^{p/2} M^{1/2}||P_{M}\theta_p ||_{L_{t}^{2}}||P_{M} g ||_{L_{t}^{2}}\nonumber\\
		&\lesssim\,\sum_{M}{\langle{N}\rangle}^{p/2}  ||P_{M} g ||_{L_{t}^{2}}\label{3llc}
		\end{align}
		where in the last inequality we used the estimate $||P_{M}\theta_p ||_{L_{t}^{2}}\,\lesssim\, || |\tau|^{-1}\varphi_{M}(\tau)||_{L_{\tau}^{2}}\,\lesssim\, M^{- 1/2} $. \\
		Now, for the second subcase,  we must have $M \sim L$. Thus we have
		\begin{align}
		\sum_{L\leq {\langle{N}\rangle}^{p}}{\langle{L+N^{p}}\rangle}^{1/2}||P_{L}(III_1) ||_{L_{t}^{2}}&\lesssim
		\sum_{L\leq \cjap{N}^{p} } \cjap{ L+N^{p} }^{1/2}\sum_{M\gtrsim L}||P_{L}(P_{\lesssim M}\theta_{p} P_{M}g) ||_{L_{t}^{2}}\nonumber\\ &\lesssim\, \sum_{L\leq \cjap{N}^{p} }\cjap{N}^{p/2}\sum_{M\sim L}||P_{L}(P_{\ll M}\theta_{p} P_{M}g) ||_{L_{t}^{2}}\nonumber\\
		&\lesssim \sum_{L\leq \cjap{N}^{p} }\cjap{N}^{p/2}||P_{\ll L}\theta_p ||_{L_{t}^{\infty}}||P_{L} g ||_{L_{t}^{2}} \nonumber\\
		&\lesssim \sum_{L}\langle{N}\rangle^{p/2} ||P_{L} g ||_{L_{t}^{2}},\label{3llc1}
		\end{align}
		and we finished this subcase and therefore the desired estimate for ($III_1$).
	
	\medskip
	\noindent
	
		{\bf Term $(III_2)$.} \,For the sum over $ L\geq {\langle{N}\rangle}^{p}$, since $ |\widehat{\theta_p}|\,\lesssim \,\langle{\xi}\rangle^p|\tau|^{-2}$, then we have by Young's inequality 
		\[
		||P_{L}(P_{M}\theta_p P_{\lesssim M}g) ||_{L_{t}^{2}}\,\lesssim\, ||\varphi_{M}\widehat{\theta_p} ||_{L_{\tau}^{1}}||P_{\lesssim M} g ||_{L_{t}^{2}}\,\lesssim\,\langle{N}\rangle^{p/2}M^{-1}||P_{\lesssim M}g ||_{L_{t}^{2}}.
		\]
		Therefore
		\begin{align}
		\sum_{L\geq {\langle{N}\rangle}^{p}}{\langle{L+N^{p}}\rangle}^{1/2}||P_{L}(III_2) ||_{L_{t}^{2}}&\lesssim
		\sum_{L\geq {\langle{N}\rangle}^{p}}\langle{L+N^{p}}\rangle^{1/2}\sum_{M\gtrsim L}||P_M\theta_pP_{\lesssim M}g ||_{L_{t}^{2}}\nonumber\\
		&\lesssim\,\sum_{L\geq {\langle{N}\rangle}^{p}} L^{1/2}{\langle{N}\rangle}^{p}||P_{\lesssim M}g ||_{L_{t}^{2}}\sum_{M\gtrsim L}M^{-1}\nonumber\\
		&\lesssim {\langle{N}\rangle}^{p/2}||P_{\lesssim M}g ||_{L_{t}^{2}}.\label{3lld}
		\end{align}
		For the case $ L\leq {\langle{N}\rangle}^{p}$, first one can see that
		\begin{align}
		\sum_{L\leq {\langle{N}\rangle}^{p}}\langle L+N^{p}\rangle^{1/2}||P_{L}(III_2) ||_{L_{t}^{2}}&\lesssim
		\sum_{L\leq \langle{N}\rangle^{p}}\langle L+N^{p}\rangle^{1/2}\sum_{M\gtrsim L}||P_L(P_M\theta_p P_{\lesssim M}g) ||_{L_{t}^{2}}\nonumber\\
		&\lesssim \sum_{L\leq \langle{N}\rangle^{p}}\langle L+N^{p}\rangle^{1/2}\sum_{M\gtrsim L}||P_L(P_M\theta_p P_{ M}g) ||_{L_{t}^{2}}\nonumber\\
		&\quad+  \sum_{L\leq \langle{N}\rangle^{p}}\langle L+N^{p}\rangle^{1/2}\sum_{M\gtrsim L}||P_L(P_M\theta P_{\ll M}g) ||_{L_{t}^{2}}\label{3lld2}
		\end{align}
		The first term has already been estimated (see \eqref{3llc}). For the second term, observing that we are in the case $\supp(\hat{g})\subset \{|\tau|\ll M\}$, thus $M\sim L$ and 
		\begin{align}
		\sum_{L\lesssim \langle{N}\rangle^{p}}\langle{L+N^{p}}\rangle^{1/2}\sum_{M\gtrsim L}||P_{L}(P_{M}\theta_{p} P_{\lesssim M}\, g) ||_{L_{t}^{2}} &\lesssim\sum_{L} {\langle{N}\rangle}^{p/2}||P_{L}\,\theta_{p} ||_{L_{t}^{2}}||P_{\lesssim L}\,  g||_{L_{t}^{\infty}}\nonumber\\
		&\lesssim \sum_{L} {\langle{N}\rangle}^{p/2}L^{-1/2}\sum_{M\ll L}M^{1/2}\|\phiN \hat{g}\|_{L^2_{\tau}}\nonumber\\
		&\lesssim \sum_{M} {\langle{N}\rangle}^{p/2}||P_{M} g||_{L_{t}^{2}},\label{3lle}
		\end{align}
		and we complete the estimate of term ($III_2$) and therefore the proof of Lemma \ref{lem-a}.
	
\end{proof}

\begin{proposition}\label{prop-b}
	Let $p>0 $ and $ \mathcal{L}: \mathcal{N}^{- p/2} \,\rightarrow\, \mathcal{S}^{- p/2}$ the linear operator defined by

		\begin{equation}\label{op-int}
		\mathcal{L}f(t, x) = \chi_{\R^{+}}(t)\eta(t)\int^{t}_{0}\emph{W}_{\!p}(t-t', t-t') f(t') dt' + \chi_{\R^{-}}(t)\eta(t)\int^{t}_{0}\emph{W}_{\!p}(t-t', t+t') f(t') dt'.
		\end{equation}

	If $f\in\mathcal{N}^{- p/2}$ , then 
	\begin{equation}\label{op}
	||\mathcal{L}f ||_{\mathcal{S}^{- p/2}}\,\lesssim\, || f ||_{\mathcal{N}^{- p/2}}
	\end{equation}
\end{proposition}
\begin{proof}
		By the definition of $\mathcal{S}\pow{-p/2}$ and $\mathcal{N}\pow{-p/2}$  it  suffices  to prove that
		\begin{equation}\label{op-a}
		||\mathcal{L}f ||_{X^{- p/2, 1/2, 1}}\,\lesssim\, || f ||_{X^{- p/2, -1/2, 1}}.
		\end{equation}
		and
		\begin{equation}\label{op-b}
		||\mathcal{L}f ||_{Y^{- p/2, 1/2}}\,\lesssim\, || f ||_{Y^{- p/2, -1/2}}.
		\end{equation}
		Noting that $\emph{W}_{\!p}(t,\,t')=e^{-t\partial_{x}^3}e^{-|t'|L_p}$, then 
		\begin{align}
		\mathcal{L}\{f\}(t,\,x)& = \eta(t)\chi_{\R^{+}}(t)\int_0^t e^{-t\partial_x^3}e^{-|t-t'|L_p}e^{t'\partial_x^3}f(t')dt'\nonumber\\
		&\quad + \eta(t)\chi_{\R^{-}}(t)\int_0^t e^{-t\partial_x^3}e^{-|t+t'|L_p}e^{t'\partial_x^3}f(t')dt'\nonumber\\
		&=e^{-t\partial_x^3}\left[\eta(t)\int_0^t \left\{\chi_{\R^{+}}(t)e^{-|t-t'|L_p}+\chi_{\R^{-}}(t)e^{-|t+t'|L_p}\right\}e^{t'\partial_x^3}f(t')dt'\right]\nonumber\\
		&=e^{-t\partial_x^3}\left[\eta(t)\int_0^t e^{-(|t|-t')L_p}e^{t'\partial_x^3}f(t')dt'\right].\label{op-c}
		\end{align}
		Writing $\omega(t',\,x)=e^{t'\partial_x^3}f(t',\,x)$, and observe that 
		\begin{align}
		e^{-(|t|-t')L_p}\omega(t',\,x)&=\mathcal{F}_{\xi}^{-1}\left( e^{-|t||\xi|^p}e^{t'|\xi|^p}\mathcal{F}_{t'}^{-1}\{\tilde{\omega}(\tau,\,\xi)\} \right)\nonumber\\
		&=\int_{\R^2}e^{(i\tau+|\xi|^p)t'}e^{i x\xi}e^{-|t||\xi|^p}\tilde{\omega}(\tau,\,\xi)d\tau d\xi,\label{op-d}
		\end{align}
		then we can conclude that
		\begin{align}
		\mathcal{L}\{f\}(t,\,x) &= e^{-t\partial_x^3}\left[   \eta(t)\int_{\R^2}e^{ix\xi}\frac{e^{it\tau}e^{(t-|t|)|\xi|^p}-e^{-|t||\xi|^p}}{i\tau+|\xi|^p}\tilde{\omega}(\tau,\,\xi)d\tau d\xi \right]\nonumber\\
		&= e^{-t\partial_x^3}\left[\mathcal{F}_{\xi}^{-1} \left\{\eta(t)  \int_{\R}\frac{e^{it\tau}e^{(t-|t|)|\xi|^p}-e^{-|t||\xi|^p}}{i\tau+|\xi|^p}\tilde{\omega}(\tau,\,\xi)d\tau \right\}  \right].\label{op-e}
		\end{align}
		The estimate \eqref{op-a} follows from Proposition \ref{prop-b}, noting that
		\begin{align}
		\|\PN\QL\mathcal{L}\{f\}\|_{L^2_{x,t}}&=\left\|\phiN(\xi)\varphi_L(\tau-\xi^3)\mathcal{F}_t\left\{\eta(t) \int_{\R}\frac{e^{it\tau}e^{(t-|t|)|\xi|^p}-e^{-|t||\xi|^p}}{i\tau+|\xi|^p}\tilde{\omega}(\tau,\,\xi)d\tau  \right\}(\tau-\xi^3)\right\|_{L^2_{\tau,\xi}}\nonumber\\
		&=\|\phiL(\tau-\xi^3)\mathcal{F}_t(\kappa_{p,\,\xi})(\tau-\xi^3)\|_{L^2_{\tau,\xi}} \nonumber\\
		&=\|\|P_L\kappa_{p,\,\xi}\|_{L^2_t}\|_{L^2_{\xi}},
		\end{align}
		where $\kappa_{ p,\,\xi}(t)$ was defined in Lemma \ref{lem-a}.\\
		Now we establish the estimate \eqref{op-b}. It suffices to prove that 
		\begin{equation}\label{op-f}
		\|(\partial_t+\partial_x^3+L_p+I)\PN \mathcal{L}\{f\}\|_{L^1_tL^2_x}\lesssim \|\PN f\|_{L^1_tL^2_x}.
		\end{equation}
		After this, squaring and summing in $N$, we get the estimate \eqref{op-b}. In order to prove \eqref{op-f}, because $\|g\|_{L^1_tL^2_x}=\|g\|_{L^1_{t\geq 0}L^2_x}+ \|g\|_{L^1_{t< 0}L^2_x}$,
		we will treat the cases $t> 0$ and $t<0$ separately.  So, in the first case, using \eqref{op-int}, one can see that
		\begin{align}
		(\partial_t+\partial_x^3+L_p+I)(\mathcal{L}\{f\})&=\partial_t\!\left(\!\eta(t)\!\!\int_0^t\!\!\emph{W}_{\!p}(t-t',t-t')f(t')dt'\!\right)+ \eta(t)\!\!\int_0^t\!(\partial_x^3+L_p)\emph{W}_{\!p}(t-t',t-t')f(t')dt'\nonumber\\
		&\quad +\eta(t)\int_0^t\!\emph{W}_{\!p}(t-t',t-t')f(t')dt'\nonumber\\
		&= \eta(t) f(t)+(\eta'(t)+\eta(t))\!\!\int_0^tS_p(t-t')f(t')dt'\nonumber\\
		&\qquad \quad+ \eta(t)\int_0^t(\partial_t+\partial_x^3+L_p)S_p(t)S_p(-t')f(t')dt'\nonumber\\
		&=\eta(t)f(t)+(\eta'(t)+\eta(t))\int_0^tS_p(t-t')f(t')dt',\label{op-h}
		\end{align}
		where  the last line was obtained by remembering that $\emph{W}_p(t,t)=S_p(t)$ ($t>0$) and  $\partial_tS_p(t)F=-(\partial_x^3+L_p)S_p(t)F$. Thus, we have 
		\begin{align}
		\|(\partial_t+\partial_x^3+L_p+I)\PN( \mathcal{L}\{f\})\|_{L^1_{t}L^2_x}&\lesssim \|\PN f\|_{L^1_tL^2_x}+\|\eta'+\eta\|_{L^1_t}\int_0^{+\infty}\|\phiN(\xi) \hat{f}(t',\,\xi)\|_{L^2_{\xi}}dt'\nonumber\\
		&\lesssim\|\PN f\|_{L^1_tL^2_x},\label{op-h1}
		\end{align}
		the desired estimate. Now we treat the case $t<0$. As mentioned in \cite{MV}, this  is  harder than the former case, because the presence of $\emph{W}_{\!p}(t-t',t+t')$ implies that $\mathcal{L}\{f\}$ does not satisfy the same equation for negative times. Indeed, with $t<t'<0$,  we have  $\emph{W}_{\!p}(t-t',\,t+t')=e^{-t(\partial_x^3-L_p)}e^{t'(\partial_x^3+L_p)}$ and  $e^{-t(\partial_x^3-L_p)}$ is the semi-group associated to another PDE: $(\partial_t+\partial_x^3-L_p)u=0$. In order to avoid this problem we decompose 
		\begin{equation}\label{decomp}\partial_t+\partial_x^3+L_p+I=(\partial_t+\partial_x^3-L_p+I)+2L_p.
		\end{equation}
		With this in hands, first we see that
		\begin{align}
		(\partial_t+\partial_x^3-L_p+I)(\mathcal{L}\{f\})&=\partial_t\!\left(\!\eta(t)\!\!\int_0^t\!\!\emph{W}_{\!p}(t-t',t-t')f(t')dt'\!\right)+ \eta(t)\!\!\int_0^t\!(\partial_x^3-L_p)\emph{W}_{\!p}(t-t',t+t')f(t')dt'\nonumber\\
		&\quad +\eta(t)\int_0^t\!\emph{W}_{\!p}(t-t',t+t')f(t')dt'\nonumber\\
		&= \eta(t)\emph{W}_{\!p}(0,\,2t) f(t)+(\eta'(t)+\eta(t))\!\!\int_0^t\emph{W}_{\!p}(t-t',t+t')f(t')dt'\nonumber\\
		&\qquad \quad+ \eta(t)\int_0^t(\partial_t+\partial_x^3-L_p)e^{-t(\partial_x^3-L_p)}e^{t'(\partial_x^3+L_p)}f(t')dt'\nonumber\\
		&=\eta(t)\emph{W}_{\!p}(0,\,2t)f(t)+(\eta'(t)+\eta(t))\int_0^t\emph{W}_{\!p}(t-t',t+t')f(t')dt'.\label{op-i}
		\end{align}
		Thus, doing the same calculations as in \eqref{op-h1} one can see that
		\begin{align}\label{op-j}
		\|(\partial_t+\partial_x^3-L_p+I)(\PN\mathcal{L}\{f\})\|_{L^1_tL^2_x}\lesssim \|\PN f\|_{L^1_tL^2_x}.
		\end{align}
		So it remains to prove a similar estimate for the term $ 2L_p(\PN \mathcal{L}\{f\})$. 
		First, we observe that $$ \|L_p(\PN \mathcal{L}\{f\})\|_{L^1_{t<0}L^2_x}\sim N^p\|\PN \mathcal{L}\{f\}\|_{L^1_{t<0}L^2_x}.$$
		Denoting by $\Theta$ the right-hand side of \eqref{op-i} we can see that %
		\begin{equation}\label{op-k}
		\PN\mathcal{L}\{f\}=-(\partial_t+\partial_x^3-L_p)(\PN\mathcal{L}\{f\})+\PN\Theta.
		\end{equation}
		Thus, integrating by parts and using Cauchy-Schwarz inequality we obtain
		\begin{align}
		\|\PN\mathcal{L}\{f\}\|^2_{L^2_x}&= \big{\langle} \PN \mathcal{L}\{f\},\,\PN\mathcal{L}\{f\}\big{\rangle}_{L^2_x}\nonumber\\
		&=-\frac{1}{2}\frac{d}{dt}\|\PN\mathcal{L}\{f\}\|_{L^2_x}^2+\big{\langle} L_p(\PN\mathcal{L}\{f\}),\,\PN\mathcal{L}\{f\}\big{\rangle}_{L^2_x} +\big{\langle} \PN\Theta,\,\PN\mathcal{L}\{f\}\big{\rangle}_{L^2_x}\nonumber\\
		&\gtrsim -\|\PN\mathcal{L}\{f\} \|_{L^2_x}\frac{d}{dt}\|\PN\mathcal{L}\{f\}\|_{L^2_x}+N^p\| \PN \mathcal{L}\{f\}\|_{L^2_x}^2 -\|\PN\Theta\|_{L^2_x}\|\PN\mathcal{L}\{f\}\|_{L^2_x}.
		\end{align}
		Therefore,
		\begin{equation}\label{op-l}
		N^p\| \PN \mathcal{L}\{f\}\|_{L^2_x}^2\lesssim\|\PN\mathcal{L}\{f\} \|_{L^2_x}^2+ \|\PN\mathcal{L}\{f\} \|_{L^2_x}\frac{d}{dt}\|\PN\mathcal{L}\{f\}\|_{L^2_x}+\|\PN\Theta\|_{L^2_x}\|\PN\mathcal{L}\{f\}\|_{L^2_x}.
		\end{equation}
		Now, for $t<0$ such that $\|\PN\mathcal{L}\{f\}\|_{L^2_x}\neq 0$, we can divide both sides in \eqref{op-l} by $\|\PN\mathcal{L}\{f\}\|_{L^2_x}$ to obtain
		\begin{equation}\label{op-m}
		N^p\| \PN \mathcal{L}\{f\}\|_{L^2_x}\lesssim\|\PN\mathcal{L}\{f\} \|_{L^2_x}+ \frac{d}{dt}\|\PN\mathcal{L}\{f\}\|_{L^2_x}+\|\PN\Theta\|_{L^2_x}.
		\end{equation}
		But, this last inequality is still true for $t<0$ such that $\|\PN\mathcal{L}\{f\}\|_{L^2_x}=0$. Indeed, $t\mapsto \|\PN\mathcal{L}\{f\}\|_{L^2_x}$ is non negative and so $(d/dt)\|\PN\mathcal{L}\{f\}\|_{L^2_x}=0$ whenever $\|\PN\mathcal{L}\{f\}\|_{L^2_x}=0$. Therefore, \eqref{op-m} is valid for all $t<0$. Integrating this inequality on $]t,\,0[$ we get
		\begin{equation}\label{op-n}
		N^p\int_t^0\| \PN \mathcal{L}\{f\}\|_{L^2_x}dt'\lesssim\int_t^0\|\PN\mathcal{L}\{f\} \|_{L^2_x}dt' -\|\PN\mathcal{L}\{f\}\|_{L^2_x}+\int_t^0\|\PN\Theta\|_{L^2_x}dt',
		\end{equation}
		and so 
		\begin{equation}\label{op-o}
		\|L_p(\PN \mathcal{L}\{f\})\|_{L^1_{t<0}L^2_x}\sim N^p\| \PN \mathcal{L}\{f\}\|_{L^1_tL^2_x}\lesssim\|\PN\mathcal{L}\{f\} \|_{L^1_tL^2_x}+\|\PN\Theta\|_{L^1_tL^2_x},
		\end{equation}
		and we finish this case noting that, obviously, $\|\PN\mathcal{L}\{f\} \|_{L^2_xL^1_t}\lesssim \|\PN f\|_{L^2_xL^1_t}$.

\end{proof}
%

%%%%%%%%%%%%%%%%%%%%%%%%%FIm lineares%%%%%%%%%%%%%%%%%%%%%%%%%%%%%%%%%

	\section{Bilinear Estimates}
	
		In this section, we will need the elementary results in the Appendix and here we establish the following important estimate:
	
	\begin{proposition}\label{bilinear} For all $u,v\in \mathcal{S}^{\pow{-p/2}}$, with $p\geq2$, we have
		\begin{equation}\label{est_bil}
		\normaNp{\partial_x(uv)}\lesssim \normaSp{u}\normaSp{v}.
		\end{equation}
		
	\end{proposition}

	\begin{proof} Using dyadic decomposition, one can write the left-hand side of \eqref{est_bil} as 
		\begin{equation}\label{est_bil01}
		\normaNp{\partial_x(uv)}^2\sim   \sum_{N}\normaNp{\sum_{N_1,\,N_2}\PN\partial_x(\PNU u\,\PND v)}^2.
		\end{equation}	

		Now, via $\FT_x$, because $ |\xi|\sim N$, $|\xi_1|\sim N_1$ and $|\xi_2|\sim N_2$, where $\xi=\xi_1+\xi_2$ (by convolution), one can see that  $\PN\partial_x(\PNU u\,\PND v)$ vanishes unless one of the following cases holds:

		\begin{description}
			\item[high-low interaction] $N\sim N_2$ and $N_1\lesssim N$;
			\item[low-high interaction] $N\sim N_1$ and $N_2\lesssim N$;
			\item[high-high interaction] $N\ll N_1\sim N_2$.
		\end{description}
		Thus, we have
		\begin{align}
		\eqref{est_bil01}  \lesssim {}&  \sum_N\normaNp{\sum_{N\ll N_1\sim N_2}\PN\partial_x(\PNU u\,\PND v)}^2 +\sum_N\normaNp{\sum_{N_2\sim N}\sum_{N_1\lesssim N_2}\PN\partial_x(\PNU u\,\PND v)}^2 \nonumber \\
		\, {}& +  \sum_N\normaNp{\sum_{N_1\sim N}\sum_{N_2\lesssim N_1}\PN\partial_x(\PNU u\,\PND v)}^2      .\label{est_bil02}	
		\end{align}
		The first sum is 
		\begin{align}
		\sim {}& \normaNp{\sum_N \sum_{N_1\gg N_1}\PN\partial_x(\PNU u\,\PNU v)}^2 \sim\normaNp{\sum_{N_1}\sum_{N\ll N_1}\PN\partial_x(\PNU u\,\PNU v)}^2  \lesssim  \sum_{N_1} \normaNp{P_{\ll N_1}\partial_x(\PNU u\,\PNU v)}^2. \label{est_bil03}
		\end{align}
		The second sum is 
		\begin{align}
		\lesssim {}& \sum_N\normaNp{\sum_{N_1\lesssim N}\PN\partial_x(\PNU u\,\PN v)}^2\nonumber\lesssim  \sum_N \normaNp{\PN\partial_x(P_{\lesssim N} u\,\PN v)}^2. \label{est_bil04}
		\end{align}
		By simetry with the second sum, we have analogous bound for the third sum. So, taking account these estimates, in order to prove \eqref{est_bil}, we need to prove the following estimates
		\begin{equation}\tag{HH}\label{HH}
		\normaNp{P_{\ll N_1}\partial_x(\PNU u\,\PNU v)} \lesssim \normaSp{\PNU u}\normaSp{\PNU v}, \ \ \ \forall \ N_1 \ \textrm{dyadic},
		\end{equation}
		\begin{equation}\tag{HL}\label{HL}
		\normaNp{\PN\partial_x(P_{\lesssim N} u\,\PN v)}\lesssim \normaSp{u}\normaSp{\PN v},\ \ \ \forall\ N\ \textrm{dyadic},
		\end{equation}
		and a similar estimate for the (symmetric) case \textbf{low-high}.\\
		
		First, we start to prove the 

		\subsection{ (HL) - estimate}
		\ \\
		
		We can see that
		\begin{equation}\label{est_bil05}
		\normaNp{\PN\partial_x(P_{\lesssim N} u\,\PN v)}\lesssim \sum_{N_1\lesssim\min\{1,N\}}\normaNp{\PN\partial_x(\PNU u\,\PN v)} +\sum_{1\lesssim N_1\lesssim N}\normaNp{\PN\partial_x(\PNU u\,\PN v)}.
		\end{equation}
		Using the H\"{o}lder and Bernstein inequalities and remember that $p\geq2$, the first sum  is
		\begin{align}
		\, \lesssim {}& \sum_{N_1\lesssim \min\{1,N\}}\normaYp{\PN\partial_x(\PNU u\,\PN v)}\lesssim \sum_{N_1\lesssim 1} \cjap{N}^{-p/2}N\|\PNU u\,\PN v\|_{\!_{L^1_tL^2_x}}\lesssim \sum_{N_1\lesssim 1} \|\PNU u\|_{\!_{L^2_tL^{\infty}_x}}\!\|\PN v\|_{\!_{L^2_tL^2_x}}   \nonumber \\
		\, \lesssim {}& \sum_{N_1\lesssim 1}N_1^{1/2}\|\PNU u\|_{\!_{L^2}}\|\PN v\|_{\!_{L^2}} \lesssim \|u\|_{\!_{L^2}}\|\PN v\|_{\!_{L^2}} \sum_{N_1\lesssim 1} N_1^{1/2}. \label{est_bil06}
		\end{align}
		Now, for the second sum we need to work a little bit more. Decomposing the bilinear term as
		\begin{equation}\label{est_bil07}
		\PN\partial_x(P_{\lesssim N}u\,\PN v)=\sum_{1\lesssim N_1\lesssim N}\, \, \sum_{L,\,L_1,\,L_2}\PN\QL\partial_x(\PNU\QLU u\,\PN\QLD v),
		\end{equation}
		via $\FT$, because $|\xi-\tau|\sim L$, $|\xi_1-\tau_1|\sim L_1$ and $ |\xi_2-\tau_2|\sim L_2$, where $\xi=\xi_1+\xi_2$ and $\tau=\tau_1+\tau_2$, and remembering the resonance relation
		\begin{equation}\label{est_bil08}
		\xi^3-\tau=(\xi_1+\xi_2)^3-(\tau_1+\tau_2)=\xi_1^3-\tau_1 +\xi_2-\tau_2+3\xi\xi_1\xi_2,
		\end{equation}
		we can conclude that the right-hand side in \eqref{est_bil07} vanishes unless $ \max\{N^2N_1,\,L_{\mathrm{med}}\}\leq L_{\max}\leq 5\max\{N^2N_1,\,L_{\mathrm{med}}\}$, i.e.,
		\begin{equation}\label{est_bil09}
		L_{\max}\sim \max\{N^2N_1,\,L_{\mathrm{med}}\},
		\end{equation}
		where $L_{\max}=\max\{L,\,L_1,\,L_2\}$, $L_{\min}=\min\{L,\,L_1,\,L_2\}$ and $L_{\mathrm{med}}\in \{L,\,L_1,\,L_2\}\backslash \{L_{\max},\,L_{\min}\}$.
		
		In view of \eqref{est_bil09}, we divide the proof in three cases, depending on the $L_{\max}$.
		
		\subsubsection{$L_{\max}=L$}\label{maxL}
		\ \\
		
		In this case we have that $L\gtrsim N^2N_1$. Thus, by the Bernstein inequality \eqref{Bernst2} and remember the estimates  \eqref{elem02} and \eqref{elem04}  we have that the second sum in \eqref{est_bil05} is
			\begin{align}
			\lesssim {}& \sum_{\smath{1\lesssim N_1\lesssim N}}\ \sum_{\smath{L\gtrsim N^2N_1}}\normaXp{\PN\QL\partial_x(\PNU u\,\PN v)}    \nonumber\\
			\lesssim {}& \sum_{\smath{1\lesssim N_1\lesssim N}}\ \sum_{\smath{L\gtrsim N^2N_1}} \cjap{N}^{-p/2}\cjap{L+N^p}^{-1/2}N\|\PNU u \, \PN v\|_{\!_{L^2}}    \nonumber\\
			\lesssim {}& N^{-p/2+1} \sum_{\smath{1\lesssim N_1\lesssim N}}\|\PNU u \, \PN v\|_{\!_{L^2}}\ \sum_{\smath{L\gtrsim N^2N_1}}L^{-1/2}     \nonumber  \\
			\lesssim {}&  N^{-p/2+1} \sum_{\smath{1\lesssim N_1\lesssim N}}\|\PNU u\|_{\!_{L^{\infty}_tL^2_x}}\|\PN v\|_{\!_{L^2}} N^{-1}{N_1}^{-1/2} N^{1/2}     \nonumber\\
			\lesssim {}& N^{-p/2} \sum_{\smath{1\lesssim N_1\lesssim N}}N_1^{p/2-1/2}\|u\|_{\!_{L^{\infty}_tH\pow{-p/2}_x}} N^{1/2}\|\PN v\|_{\!_{L^2}}     \nonumber\\
			\lesssim {}& \normaSp{u}\normaSp{\PN v}  \sum_{\smath{1\lesssim N_1\lesssim N}}(N_1/N)^{p/2-1/2} \nonumber
			\end{align}

		and this establishes the desired estimates, noting that the sum converge and is $\lesssim 1$, because $p\geq 1$.
		
		\subsubsection{$L_{\max}=L_1$}\label{maxL1}
		\ \\

		In this case, we have  $L_1\sim N^2N_1$ or $L_1\sim L_{\textrm{med}}$. The latter case implies that $L_1\sim L_{\textrm{med}}\gtrsim N^2N_1$ and thus, we have two subcases
		\begin{enumerate}%[(i)]
			\item $L_{\textrm{med}}=L$, and so $L_{\max}=L_1\sim L\gtrsim N^2N_1$;
			\item $L_{\textrm{med}}\neq L$, and so $L_{\textrm{med}}=L_2$ and $L_1\sim L_2\gtrsim N^2N_1$.
		\end{enumerate}
		Therefore, we have that the second sum in \eqref{est_bil05} is
		\begin{align}
			\lesssim\!\sum_{\smath{1\lesssim N_1\lesssim N}} \sum_{\ \smath{L_1\!\sim\! N^2N_1}}\!\normaNp{\PN\partial_x(\PNU \QLU u\,\PN  v)}    +\! \!\sum_{\smath{1\lesssim N_1\lesssim N}} \sum_{\ \smath{L_{1}\!\sim\! L_{2}\!\gtrsim\! N^2\!N_1}} \!\!\normaNp{\PN\partial_x(\PNU \QLU u\,\PN \QLD v)} \hspace{2.3cm}\nonumber \\
			\ \  + \sum_{\smath{1\lesssim N_1\lesssim N}} \sum_{\ \smath{ L\!\gtrsim\! N^2\!N_1}}\!\normaNp{\PN\QL\partial_x(\PNU u\,\PN  v)}.\hspace{7cm}\label{est_bil10}
			\end{align}
			The last sum was treated in subsection \ref{maxL}. Now, by H\"{o}lder and Bernstein inequalities and remembering the estimates \eqref{elem02} and \eqref{elem03}, the first sum in  \eqref{est_bil10} is
		\begin{align}
		\, \lesssim {}& \sum_{\smath{1\lesssim N_1\lesssim N}}\! \cjap{N}^{-p/2} N\|\PNU Q_{N^2\!N_1}u\, \PN v\|_{\!_{L^1_tL^2_x}}\lesssim \sum_{\smath{1\lesssim N_1\lesssim N}}\! \cjap{N}^{-p/2} N\|\PNU Q_{N^2\!N_1}u\|_{\!_{L^2_tL^{\infty}_x}} \|\PN v\|_{\!_{L^2}} \nonumber\\
		\lesssim {}& \sum_{\smath{1\lesssim N_1\lesssim N}}\! \cjap{N}^{-p/2} (N^2N_1)^{1/2}\|Q_{N^2\!N_1}\PNU u\|_{\!_{L^2_tL^{2}_x}} \|\PN v\|_{\! \mathcal{S}\pow{-p/2}}     \nonumber\\
		\lesssim {}& \sum_{\smath{1\lesssim N_1\lesssim N}}\! \left(N_1/N\right)^{p/2} \left\{\sum_{L}\bigl[L^{1/2}\|Q_{L}\cjap{N_1}^{-p/2}\PNU u\|_{\!_{L^2_tL^{2}_x}}\bigr]^2\right\}^{1/2} \|\PN v\|_{\! \mathcal{S}\pow{-p/2}} \nonumber\\
		\lesssim {}& \sum_{\smath{1\lesssim N_1\lesssim N}} \left(N_1/N\right)^{p/2}\| \cjap{N_1}^{-p/2}\PNU u\|_{\!\mathcal{S}\pow {0}}\|\PN v\|_{\! \mathcal{S}\pow{-p/2}} \lesssim \| u\|_{\!\mathcal{S}\pow {-p/2}}\|\PN v\|_{\! \mathcal{S}\pow{-p/2}} \sum_{\smath{1\lesssim N_1\lesssim N}} \left(N_1/N\right)^{p/2}, \nonumber
		\end{align}
		noting that the sum above is $\lesssim 1$, because $p>0$.
		
		\ \\
		%\ \ \ \ \ {\textbf{ERRADO}}\ \ \ \ \ \\
		
		It remains to treat the second sum in \eqref{est_bil10}. Arguing as before, this term is 
		\begin{align}
		\lesssim {}& \sum_{\smath{1\lesssim N_1\lesssim N}} \sum_{\ \smath{L_{1}\!\gtrsim\! N^2\!N_1}} \normaYp{\PN\partial_x(\PNU \QLU u\,\PN \QLU v)} \lesssim \sum_{\smath{1\lesssim N_1\lesssim N}} \sum_{\ \smath{L_{1}\!\gtrsim\! N^2\!N_1}} N^{-p/2+1} N_1^{1/2}\| \PNU\QLU u\|_{\!_{L^2}}\|\PN\QLU v\|_{\!_{L^2}}. \label{est_bil11}
		\end{align}
		Also, is clear that $\| \PNU\QLU u\|_{\!_{L^2}}\|\PN\QLU v\|_{\!_{L^2}}$ is $\lesssim \|\PNU u\|_{\!\mathcal{S}\pow{-p/2}}\|\PN v\|_{\!\mathcal{S}\pow{-p/2}}$. But, on the other hand this same product is $\lesssim L_1^{-1}\|\PNU u\|_{\!\mathcal{S}\pow 0}\|\PN v\|_{\!\mathcal{S}\pow 0}$. So, if $\theta\in[0,\,1]$
		\begin{equation}\label{est_bil12}
		\| \PNU\QLU u\|_{\!_{L^2}}\|\PN\QLU v\|_{\!_{L^2}}\lesssim L_1^{-\theta}\|\PNU u\|_{\!\mathcal{S}\pow 0}^{\theta}\|\PNU u\|_{\!\mathcal{S}\pow{-p/2}}^{1-\theta}\|\PN v\|_{\!\mathcal{S}\pow 0}^{\theta}\|\PN v\|_{\!\mathcal{S}\pow{-p/2}}^{1-\theta}.
		\end{equation}
		Using this estimate with $\theta=1/2$ and localization properties (see the Appendix), we have
\begin{equation}\label{est_bil12}
		\| \PNU\QLU u\|_{\!_{L^2}}\|\PN\QLU v\|_{\!_{L^2}}\lesssim L_1^{-1/2}N_1^{p/4}N^{p/4}\|P_{N_1} u\|_{\!\mathcal{S}\pow{ -p/2}}\|P_{N} v\|_{\!\mathcal{S}\pow{-p/2}}.
		\end{equation}

 Considering this inequality in  the right-hand side of \eqref{est_bil11} we get
		\begin{align}
		\lesssim {}& \sum_{\smath{1\lesssim N_1\lesssim N}} \sum_{\ \smath{L_{1}\!\gtrsim\! N^2\!N_1}} N^{-p/2+1} N_1^{1/2}L_1^{-1/2}N_1^{p/4} N^{p/4}\|\PNU u\|_{\!\mathcal{S}\pow{-p/2}}\|\PN v\|_{\!\mathcal{S}\pow{-p/2}} \nonumber\\
		\lesssim{}& N^{-p/4+1}  \|\PN v\|_{\!\mathcal{S}\pow{-p/2}} \sum_{\smath{1\lesssim N_1\lesssim N}}N_1^{1/2+p/4} (N^2\!N_1)^{-1/2} \| P_{N_1}u\|_{\!\mathcal{S}\pow{-p/2}}
\nonumber\\
\lesssim{}& N^{-p/4}  \|\PN v\|_{\!\mathcal{S}\pow{-p/2}} \sum_{\smath{1\lesssim N_1\lesssim N}}N_1^{p/4} \| P_{N_1}u\|_{\!\mathcal{S}\pow{-p/2}}
\nonumber\\
\lesssim{}& N^{-p/4}  \|\PN v\|_{\!\mathcal{S}\pow{-p/2}} \left(\sum_{\smath{1\lesssim N_1\lesssim N}}N_1^{p/2}\right)^{1/2}      \left(\sum_{\smath{1\lesssim N_1\lesssim N}}\| P_{N_1}u\|_{\!\mathcal{S}\pow{-p/2}}^{2}\right)^{1/2}
\nonumber\\
\lesssim{}&  \|\PN v\|_{\!\mathcal{S}\pow{-p/2}} \left(\sum_{\smath{1\lesssim N_1\lesssim N}}\| P_{N_1}u\|_{\!\mathcal{S}\pow{-p/2}}^{2}\right)^{1/2},
		\end{align}
		and using the localization property \eqref{locprop} the desired estimate follows.

		\subsubsection{$L_{\max}=L_2$}\label{maxL2}
		\ \\
		
		In this case we have $L_2\sim N^2N_1$ or $L_2\sim L_{\textrm{med}}$. The latter case implies that $L_2\sim L_{\textrm{med}}\gtrsim N^2N_1$, and thus we have two subcases
		\begin{enumerate}%[(i)]
			\item $L_{\textrm{med}}=L_1$, and so $L_{1}\sim L_2 \gtrsim N^2N_1$;
			\item $L_{\textrm{med}}\neq L_1$, and so $L_{\textrm{med}}=L$ and $L_2\sim L\gtrsim N^2N_1$.
		\end{enumerate}
		Therefore, we have that 
		\begin{align}
		\sum_{\smath{1\lesssim N_1\lesssim N}}\normaNp{\PN\partial_x(\PNU u\,\PN v)}  \sim  \sum_{\smath{1\lesssim N_1\lesssim N}} \normaNp{\sum_{\smath{L,\,L_1,\,L_2}}\PN\QL\partial_x(\PNU \QLU u\,\PN \QLD v)}  \hspace{-2.5cm}  {}& \nonumber\\
		\lesssim\!\sum_{\smath{1\lesssim N_1\lesssim N}} \sum_{\ \smath{L_2\!\sim\! N^2N_1}}\!\normaNp{\PN\partial_x(\PNU u\,\PN \QLD v)}  {}&  + \!\sum_{\smath{1\lesssim N_1\lesssim N}} \sum_{\ \smath{L_{1}\!\sim\! L_{2}\!\gtrsim\! N^2\!N_1}} \!\normaNp{\PN\partial_x(\PNU \QLU u\,\PN \QLD v)} \nonumber \\
		\ \  + \sum_{\smath{1\lesssim N_1\lesssim N}} \sum_{\ \smath{ L\!\gtrsim\! N^2\!N_1}} \normaNp{\PN\QL\partial_x(\PNU u\,\PN  v)}. \label{est_bil13}
		\end{align}
		The last two sums were treated in subsection \ref{maxL1} item (ii) and  in subsection \ref{maxL}, respectively. For the first sum, as before, we can obtain that
		\begin{align}
		\sum_{\smath{1\lesssim N_1\lesssim N}} \sum_{\ \smath{L_2\!\sim\! N^2N_1}}\!\normaNp{\PN\partial_x(\PNU u\,\PN \QLD v)}\lesssim {}& \sum_{\smath{1\lesssim N_1\lesssim N}}N^{-p/2+1}N_1^{1/2}\|\PNU u\|_{\!_{L^2}}\|\PN Q_{N^2\!N_1} v\|_{\!_{L^2}}    \nonumber\\
		\lesssim {}& \sum_{\smath{1\lesssim N_1\lesssim N}}N^{-p/2}\| \PN u\|_{\!\mathcal{S}\pow{-p/2}}\bigl[(N^2N_1)^{1/2} \|\PN Q_{N^2\!N_1} v\|_{\!_{L^2}}  \bigr]   \nonumber \\ 
		\lesssim {}& N^{-p/2}\left\{\sum_{\smath{1\lesssim N_1\lesssim N}} \|\PN u\|^2_{\!\mathcal{S}\pow{-p/2}}  \right\}^{1/2}\!\left\{ \sum_{\smath{N^2\!\lesssim L\lesssim \!N^3}} \bigl[   L^{1/2}\|\QL\PN v\|_{\!_{L^2}} \bigr]^2  \right\}^{1/2} \nonumber\\
		\lesssim{}& \|u\|_{\!\mathcal{S}\pow{-p/2}}\|\PN v\|_{\!\mathcal{S}\pow{-p/2}}, \nonumber
		\end{align}
		and we finish the proof for the case \textbf{HL-estimate}.\\

		Now, we finish the proof of \eqref{est_bil01}, establishing the
		
		\subsection{(HH) - estimate.} 
		\ \\
		
		We can see that 
		\begin{equation}\label{est_bil14}
		\normaNp{P_{\ll N_1}\partial_x(\PNU u\,\PNU v)} \lesssim  \sum_{\smath{N\lesssim \min\{1,N_1\}}}\normaNp{\PN\partial_x(\PNU u\,\PNU v)} + \sum_{\smath{1\lesssim N\ll N_1}}\normaNp{\PN\partial_x(\PNU u\,\PNU v)}.
		\end{equation}
		Again, the estimate for the first term is easier than for the second. In fact, if $p \geq 0$, the first term is
		\begin{align}
		\lesssim {}&\sum_{\smath{N\lesssim 1}}\normaYp{\PN\partial_x(\PNU u\,\PNU v)} \lesssim \sum_{\smath{N\lesssim1}} N^{-p/2}N\|\PN(\PNU u\,\PNU v)\|_{\!_{L^1_tL^2_x}} \nonumber\\
		\lesssim {}& \sum_{\smath{N\lesssim 1}} N^{-p/2}N^{3/2}\|\PNU u \|_{\!_{L^2}} \|\PNU v\|_{\!_{L^2}} \lesssim  \normaSp{\PNU u}\normaSp{\PNU v}\sum_{\smath{N\lesssim 1}} N^{3/2} .\nonumber
		\end{align}
		For the second term, using dyadic decompostion and triangular inequality
			\begin{equation}\label{est_bil15}\sum_{\smath{1\lesssim N\ll N_1}}\normaNp{\PN\partial_x(\PNU u\,\PNU v)}\leq\sum_{\smath{1\lesssim N\ll N_1}}\, \, \sum_{\smath{L,\,L_1,\,L_2}}\normaNp{\PN\QL\partial_x(\PNU\QLU u\,\PN\QLD v)}.
			\end{equation}
		Using again the resonance relation \eqref{est_bil08} and arguing as before, we may restrict ourself to the region where 
		\begin{equation}\label{estbil16}
		L_{\max}\sim \max\{ N_1^2N,\,L_{\textrm{med}} \},
		\end{equation}
		and this leads us to consider the following three cases. By simmetry we can suppose that $L_1 \geq L_2$.

		\subsubsection{$L_{\max}=L$}
		\ \\
		
		In this case we have $L\gtrsim N_1^2N$. Also, for $\lambda>0$ to be choosen later,
			\begin{equation}
			\sum_{\smath{L\gtrsim N_1^2N}}\sum_{L_1,L_2}=\sum_{\smath{L\gtrsim N_1^2N}}\,\,\sum_{\substack{ L_1 \lesssim \lambda, \\ \textrm{or} \\L_2 \lesssim \lambda}}+\sum_{\smath{L\gtrsim N_1^2N}}\,\,\sum_{\substack{L_1\gtrsim \lambda,\\ L_2 \gtrsim \lambda}} 
			\leq \sum_{L_1 \lesssim \lambda}+\sum_{L_2 \lesssim \lambda}+\sum_{\substack{L_1\gtrsim \lambda,\\ L_2 \gtrsim \lambda}}.
			\end{equation} 
			Therefore the right-hand side of \eqref{est_bil15} is
			\begin{align}
			&\leq\sum_{\smath{1\lesssim N\ll N_1}} \sum_{\smath{L\gtrsim N_1^2N}}\,\,\sum_{\substack{ L_1 \lesssim \lambda, \\ \textrm{or} \\L_2 \lesssim \lambda}}\normaNp{\PN\QL\partial_x(\PNU\QLU u\,\PN\QLD v)}\nonumber\\
			&\hspace{3cm}+\sum_{\smath{1\lesssim N\ll N_1}}\sum_{\smath{L\gtrsim N_1^2N}}\sum_{L_1, L_2 \gtrsim \lambda,}\normaNp{\PN\QL\partial_x(\PNU\QLU u\,\PN\QLD v)} \nonumber\\
			&\leq \sum_{\smath{1\lesssim N\ll N_1}}\sum_{L_1 \lesssim \lambda}\normaNp{\PN\QL\partial_x(\PNU\QLU u\,\PN\QLD v)}+\sum_{\smath{1\lesssim N\ll N_1}}\sum_{L_2 \lesssim \lambda}\normaNp{\PN\QL\partial_x(\PNU\QLU u\,\PN\QLD v)}\nonumber\\
			&\ \ \ \ \ +\sum_{\smath{1\lesssim N\ll N_1}}\sum_{\substack{L_1\gtrsim \lambda,\\ L_2 \gtrsim \lambda}}\normaNp{\PN\QL\partial_x(\PNU\QLU u\,\PN\QLD v)}\nonumber\\
			&= \mathcal{L}_1(p)+\mathcal{L}_2(p)+\mathcal{L}_3(p).\label{est_bil17}
			\end{align}

		We will only  estimate $\mathcal{L}_1(p)$ and $\mathcal{L}_3(p)$, because the estimate of $\mathcal{L}_2(p)$ is similar to the estimate of $\mathcal{L}_1(p)$ . We consider the following three subcases:
		\\
		\\
		{\bf  Subcase I:}    If $p>3$\\
		
		Let $\lambda =N_1^{\alpha} N^{\beta}$, where $\alpha$ and $\beta$ will be choosen later. Taking advantage of the $X_p\pow{-p/2,-1/2,1}$ part of $\mathcal{N}\pow{-p/2}$ and using \eqref{Bernst3} we obtain
			\begin{align}
			\mathcal{L}_1(p) \leq & \sum_{\smath{1\lesssim N\ll N_1}}\sum_{L_1 \lesssim \lambda}\|\PN\QL\partial_x(\PNU\QLU u\,\PN\QLD v)\|_{X\pow{-p/2,-1/2,1}}\nonumber\\
			\lesssim& \sum_{\smath{1\lesssim N\ll N_1}} \sum_{\substack{L \gtrsim N_1^2N,\\ L_1 \leq \lambda}}N^{-p/2+1}L^{-1/2}\|P_N(Q_{L_1}P_{N_1 }u P_{N_1 }v) \|_{L^2}\label{eqPRX1}\\
			\lesssim & \sum_{\smath{1\lesssim N\ll N_1}} \sum_{\substack{L \gtrsim N_1^2N,\\ L_1 \leq \lambda}}N^{-p/2+1}L^{-1/2}N^{1/2}L_1^{1/2}\|P_{N_1 }u  \|_{L^2}\|P_{N_1 }v  \|_{L^2}\nonumber\\
			\lesssim & \|P_{N_1 }u  \|_{L^2}\|P_{N_1 }v  \|_{L^2} N_1^{\alpha/2-1} \sum_{\smath{1\lesssim N\ll N_1}}  N^{-p/2+1+\beta/2}.\nonumber
			\end{align}
		
		Thus for $\alpha=2$, $\beta =1-\epsilon$ and $0<\epsilon<1$, we obtain
		\begin{equation}
		\begin{split}
		\mathcal{L}_1(p)  \lesssim & \|P_{N_1 }u  \|_{L^2}\|P_{N_1 }v  \|_{L^2}   \\
		\lesssim & \|P_{N_1 }u  \|_{\mathcal{S}\pow{-p/2}}\|P_{N_1 }v  \|_{\mathcal{S}\pow{-p/2}}.
		\end{split}
		\end{equation}
For $\mathcal{L}_3(p)$, considering the norm $Y_p\pow{-p/2,-1/2}$, we have
			\begin{equation}\label{eqPRX5}
			\begin{split}
			\mathcal{L}_3(p)  \leq &\sum_{\smath{1\lesssim N\ll N_1}} \sum_{L_1, L_2 \gtrsim \lambda}\|\PN\QL\partial_x(\PNU\QLU u\,\PN\QLD v)\|_{Y\pow{-p/2,-1/2}_p}\\
			\lesssim & \sum_{\smath{1\lesssim N\ll N_1}} N^{-p/2+1}\|P_N \left( P_{N_1 }Q_{\gtrsim \lambda}u P_{N_1 }Q_{\gtrsim \lambda}v\right) \|_{L^1_t L^2_x}\\
			\lesssim & \sum_{\smath{1\lesssim N\ll N_1}} N^{-p/2+1}N^{1/2}\|P_{N_1 }Q_{\gtrsim \lambda}u \|_{L^2}\| P_{N_1 }Q_{\gtrsim \lambda}v)  \|_{L^2}\\
			\lesssim & \|P_{N_1 }u  \|_{L^2}\|P_{N_1 }v  \|_{L^2} \sum_{1\lesssim N \ll N_1}  N^{-(p+3)/2}\\
			\lesssim &  \|P_{N_1 }u  \|_{\mathcal{S}^{-p/2}}\|P_{N_1 }v  \|_{\mathcal{S}^{-p/2}}.
			\end{split}\end{equation}
			\\
		
		{\bf  Subcase II:}    If $ p = 2$ \\
		
			This case was treated in \cite{MV} also considering $\alpha=2$ and $\beta=1-\epsilon$. Indeed let $\mathcal{X}(p)$ the  right side of the second inequality in \eqref{eqPRX1}, using the Kato smoothing effect, was proved in \cite{MV} that
			\begin{equation}\label{eqPRX2}
			\begin{split}
			\mathcal{L}_1(2)  \leq &\sum_{\smath{1\lesssim N\ll N_1}}\sum_{L_1 \lesssim \lambda}\|\PN\QL\partial_x(\PNU\QLU u\,\PN\QLD v)\|_{X\pow{-1,-1/2,1}}\\
			\lesssim& \mathcal{X}(2)\\
			\lesssim &  \|P_{N_1 }u  \|_{\mathcal{S}^{-1}}\|P_{N_1 }v  \|_{\mathcal{S}^{-1}}. 
			\end{split}
			\end{equation}
			Let  $\mathcal{Y}(p)$ the right side of the second inequality in \eqref{eqPRX5}, using the inequality \eqref{elem03}, was proved in \cite{MV} that
			\begin{equation}\label{eqPRX2.1}
			\begin{split}
			\mathcal{L}_3(2) \leq & \sum_{\smath{1\lesssim N\ll N_1}} \sum_{L_1, L_2 \gtrsim \lambda}\|\PN\QL\partial_x(\PNU\QLU u\,\PN\QLD v)\|_{Y\pow{-1,-1/2}}\\
			\lesssim &\mathcal{Y}(2)\\
			\lesssim &  \|P_{N_1 }u  \|_{\mathcal{S}^{-1}}\|P_{N_1 }v  \|_{\mathcal{S}^{-1}}. 
			\end{split}
			\end{equation}
			\\

	{\bf  Subcase III:}    If $ 2<p\leq 3$\\
		
		Let $p_0=2$ and we consider $p_1>3$, therefore $p= \theta p_0+(1-\theta)p_1$, where $\theta \in (0,1)$.
		
		As above we have
		\begin{equation}\label{eqPRX3}
		\begin{split}
		\mathcal{L}_1(p)
		\lesssim\  &\  \mathcal{X}(p)=\sum_{\smath{1\lesssim N\ll N_1}} \sum_{\substack{L \gtrsim N_1^2N,\\ L_1 \leq \lambda}}N^{-p/2+1}L^{-1/2}\|P_N(Q_{L_1}P_{N_1 }u P_{N_1 }v) \|_{L^2}\\   
		= & \sum_{\smath{1\lesssim N\ll N_1}} N^{-p/2}\Bigl\{N \displaystyle\sum_{\substack{L \gtrsim N_1^2N,\\ L_1 \leq \lambda}}L^{-1/2}\|P_N(Q_{L_1}P_{N_1 }u P_{N_1 }v) \|_{L^2}\Bigr\}\\
		=& \sum_{\smath{1\lesssim N\ll N_1}} N^{-\theta p_0/2}\mathcal{H}^{\theta}N^{(1-\theta)p_1/2}\mathcal{H}^{1-\theta}.
		\end{split}
		\end{equation}
		where $$\mathcal{H}=N\!\!\!\displaystyle\sum_{\substack{L \gtrsim N_1^2N,\\ L_1 \leq \lambda}}L^{-1/2}\|P_N(Q_{L_1}P_{N_1 }u P_{N_1 }v) \|_{L^2}.$$ Using H\"older inequality (with $p=1/\theta$ and $q=1/(1-\theta)$), Case I and Case II, we arrive to
		\begin{equation}
		\begin{split}
		\mathcal{X}(p) \leq & \mathcal{X}(p_0)^{\theta}\mathcal{X}(p_1)^{1-\theta}\\
		\lesssim & \left(   \|P_{N_1 }u  \|_{\mathcal{S}^{-p_0/2}}\|P_{N_1 }v  \|_{\mathcal{S}^{-p_0/2}}  \right)^{\theta}\left(   \|P_{N_1 }u  \|_{\mathcal{S}^{-p/2}}\|P_{N_1 }v  \|_{\mathcal{S}^{-p/2}}  \right)^{1-\theta}\\
		\sim & N_1^{-p_0 \theta/2}N_1^{-p_1(1- \theta)/2} \|P_{N_1 }u  \|_{\mathcal{S}^{0}}N_1^{-p_0 \theta/2}N_1^{-p_1(1- \theta)/2} \|P_{N_1 }v  \|_{\mathcal{S}^{0}}\\
		\sim & N_1^{-p/2} \|P_{N_1 }u  \|_{\mathcal{S}^{0}}N_1^{-p/2} \|P_{N_1 }v  \|_{\mathcal{S}^{0}}\\
		\sim &  \|P_{N_1 }u  \|_{\mathcal{S}^{-p/2}}\|P_{N_1 }v  \|_{\mathcal{S}^{-p/2}}.
		\end{split}
		\end{equation}
		Similarly, the estimate for $ \mathcal{L}_3(p)$ follows using again the  Cases I and  II and the interpolation inequality
		$$
		\mathcal{Y}(p) \leq  \mathcal{Y}(p_0)^{\theta}\mathcal{Y}(p_1)^{1-\theta}, \quad p_0=2,\,\, p_1>3.
		$$
		\subsubsection{$L_{\max}=L_1$}
		\ \\
		Using the relation \eqref{estbil16}, first we consider the case $L_1 \sim N_1^2 N$ and we need to estimate 
		\begin{equation}\label{HHeqx11}
		\begin{split}
		\|  \sum_{\smath{1\lesssim N\ll N_1}} P_N \partial_x (P_{N_1}Q_{N_1^2N} u P_{N_1}v)  \|_{Y^{-p/2, -1/2}} \sim & \left(  \sum_{\smath{1\lesssim N\ll N_1}} N^{-p+2} \| P_N  (P_{N_1}Q_{N_1^2N} u P_{N_1}v) \|_{L^1_t L^2_x}^2  \right)^{1/2}\\
		\lesssim &  \left(  \sum_{\smath{1\lesssim N\ll N_1}} N^{-p+3} \| P_{N_1}Q_{N_1^2N} u\|_{L^2}^2\| P_{N_1}v \|_{L^2}^2  \right)^{1/2},
		\end{split}
		\end{equation}
		where was used the inequality \eqref{Bernst1}. Let $\mathcal{T}(p)$ the sum in the right side of the second  inequality of \eqref{HHeqx11}.  If $p>3$ and  considering the inequality \eqref{elem02} it is not hard to see that
		\begin{equation}
		\begin{split}
		\mathcal{T}(p):=&\sum_{\smath{1\lesssim N\ll N_1}} N^{-p+3} \| P_{N_1}Q_{N_1^2N} u\|_{L^2}^2\| P_{N_1}v \|_{L^2}^2\\
		\lesssim&  \| P_{N_1}v \|_{L^2}^2  \| P_{N_1}u \|_{L^2}^2  \sum_{\smath{1\lesssim N \ll N_1}} N^{-p+3}\\
		\lesssim  &  \|P_{N_1 }u  \|_{\mathcal{S}^{-p/2}}^2\|P_{N_1 }v  \|_{\mathcal{S}^{-p/2}}^2, \quad p>3.
		\end{split}
		\end{equation}
		Using a change of variable and \eqref{elem03}, was proved in \cite{MV}
		$$
		\mathcal{T}(2) \lesssim  \|P_{N_1 }u  \|_{\mathcal{S}^{-1}}^2\|P_{N_1 }v  \|_{\mathcal{S}^{-1}}^2.
		$$
		An interpolation  argument as above (see the Case $L_{max}=L$) proves the inequality in the case $2 < p \leq3$.
		
		Finally in the case  $L_1 \sim L_2 \gtrsim N_1^2N$, considering again the $Y^{-p/2,-1/2}_p$ norm we have
		\begin{equation}\label{HHeqx1}
		\begin{split}
		\|  \sum_{\smath{1\lesssim N\ll N_1}} & \sum_{L_1 \sim L_2 \gtrsim  N_1^2N} P_N \partial_x (P_{N_1}Q_{L_1} u P_{N_1}Q_{L_2}v)  \|_{Y_p^{-p/2, -1/2}}   \\
		\lesssim & \sum_{\smath{1\lesssim N\ll N_1}}  \sum_{\smath{L_1 \gtrsim  N_1^2N}}\| P_N \partial_x (P_{N_1}Q_{L_1} u P_{N_1}Q_{L_1}v)  \|_{Y_p^{-p/2, -1/2}}\\
		\lesssim &  \mathcal{K}(p):= \sum_{\smath{1\lesssim N\ll N_1}}  \sum_{\smath{L_1 \gtrsim  N_1^2N}} N^{-p/2}N \| P_N (P_{N_1}Q_{L_1} u P_{N_1}Q_{L_1}v)  \|_{L_t^1 L^2_x},
		\end{split}
		\end{equation}
		Was proved in \cite{MV} that $ \mathcal{K}(2) \lesssim  \|P_{N_1 }u  \|_{\mathcal{S}^{-1}}\|P_{N_1 }v  \|_{\mathcal{S}^{-1}}$. Therefore by the interpolation argument is sufficient to consider $p>3$. In fact using \eqref{Bernst1} and Cauchy-Schwarz inequality, we get
		\begin{equation}\label{HHeqx1.1}
		\begin{split}
		\mathcal{K}(p) \lesssim & \sum_{\smath{1\lesssim N\ll N_1}}  \sum_{L_1 \gtrsim  N_1^2N} N^{-p/2+1} N^{1/2}\| P_{N_1}Q_{L_1} u \|_{L^2} \|P_{N_1}Q_{L_1}v  \|_{ L^2}\\
		\lesssim & \sum_{\smath{1\lesssim N\ll N_1}}  N^{(-p+3)/2} \left( \sum_{L_1 }\|Q_{L_1} P_{N_1} u \|_{L^2}^2 \right)^{1/2} \left( \sum_{L_1 }\|Q_{L_1} P_{N_1} v \|_{L^2}^2 \right)^{1/2}\\
		\lesssim &  \| P_{N_1} u \|_{L^2}\| P_{N_1} v \|_{L^2} \sum_{1\lesssim N \ll N_1}  N^{(-p+3)/2}\\
		\lesssim &   \|P_{N_1 }u  \|_{\mathcal{S}^{-p/2}}\|P_{N_1 }v  \|_{\mathcal{S}^{-p/2}}.
		\end{split}
		\end{equation}

	\end{proof}

%%%%%%%%%%%%%%%%%%%%%%%%%%%%%%FIM bilineares%%%%%%%%%%%%%%%%%%%%%%%%%%%%%%%%%%%%%%%%
	
	\section{Well-posedness}\label{We-Pos}

In this section we obtain well-posedness results. In order to remove some restrictions on the size of the initial data, we change the metric of the resolution space used in the previous sections. We define the space $ \mathcal{Z}_{\beta} := \,\mathcal{S}^{- p/2}\times \mathcal{S}^0 $ and define, for $\beta\geq 1$,  the functional
\[
\| u\|_{\mathcal{Z}_{\beta}}:=  \inf_{\scriptsize\begin{array}{ll}                              
	&u= u_1+u_2 \\
	& u_1\in \mathcal{S}^{- p/2}, u_2\in \mathcal{S}^0
	\end{array}}     \left\{\, || u_1 ||_{\mathcal{S}^{- p/2}} + \dfrac{1}{\beta} || u_2 ||_{\mathcal{S}^0} \,\right\}
\]
for all $\, u\in\mathcal{Z}_{\beta}$. Which defines a new norm on $ \mathcal{S}^{- p/2} $. In addition, this norm is equivalent to $||.||_{\mathcal{S}^{- p/2}}$, i.e. $ || u ||_{\mathcal{Z}_{\beta}} \sim || u ||_{\mathcal{S}^{- p/2}},\quad  u\in \mathcal{S}^{- p/2}$.

In order to remove the restrictions said before, first we   establish the following nonlinear estimate

\begin{proposition}\label{XXav1}
	There exists $\nu > 0 $  such that for all $( u, v )\in \mathcal{S}^{0}\times \mathcal{S}^{- p/2}$  with compact support (in time) contained in $[- T, T ]$, then
	\begin{equation}\label{nao-lin}
	\| \partial_{x}( u v ) \|_{\mathcal{N}^{- p/2}} \lesssim T^{\nu} \| u \|_{\mathcal{S}^0} \| v \|_{\mathcal{S}^{- p/2}}.
	\end{equation}
\end{proposition}

\begin{remark}
	For any $\theta >0$, there exists $\mu = \mu(\theta) > 0$ such that for any smooth function $f$ with compact support in time in $[-T,T]$. We have to
	\begin{equation}\label{dsg}
	\left\| \mathcal{F}_{\tau,\xi}^{-1}\left\{ \dfrac{\hat{f}(\tau,\xi)}{\langle{\tau - \xi^3}\rangle}\right\}\right\|_{L^{2}_{t,x}} \lesssim T^{\mu} \| f \|_{L^{2}_{t,x}}.
	\end{equation}
\end{remark}
The proof of this remark can be found in \cite{GTV} (Lemma 3.1) and \cite{MR} (Lemma 3.6).
\begin{proof}[Proof of Proposition \ref{XXav1}]
	This proof is very similar with the proof of the Proposition \ref{bilinear}. For the sackness of completes we will proof the proposition in the more difficult case: (H L)\\
	i) $L_{max}= L$.  In this case observe that
	\begin{align}
	\sum_{\smath{1\lesssim N_1\lesssim N}}\normaNp{\PN\partial_x(\PNU u\,\PN v)}\lesssim {}& \sum_{\smath{1\lesssim N_1\lesssim N}}\ \sum_{\smath{L\gtrsim N^2N_1}}\normaXp{\PN\QL\partial_x(\PNU u\,\PN v)}    \nonumber\\
	\lesssim {}& \sum_{\smath{1\lesssim N_1\lesssim N}}\ \sum_{\smath{L\gtrsim N^2N_1}} \cjap{N}^{-p/2}\cjap{L+N^p}^{-1/2}N\|\PNU u \, \PN v\|_{\!_{L^2}}    \nonumber\\
	\lesssim {}& N^{-p/2+1} \sum_{\smath{1\lesssim N_1\lesssim N}}\|\PNU u \, \PN v\|_{\!_{L^2}}\ \sum_{\smath{L\gtrsim N^2N_1}}L^{-1/2}     \nonumber  \\
	\lesssim {}&  N^{-p/2+1} \sum_{\smath{1\lesssim N_1\lesssim N}}\|\PNU u\|_{\!_{L^2_tL^{\infty}_x}}\|\PN v\|_{\!_{L^\infty_t L^2_x}} N^{-1}{N_1}^{-1/2}      \nonumber\\
	\lesssim {}& \sum_{\smath{1\lesssim N_1\lesssim N}}N_1^{-1/2}\|\PN u\|_{\!_{L^2_tH^{3/4}_x}} \|\PN v\|_{\!_{L^{\infty}_tH\pow{-p/2}_x}}     \nonumber\\
	\lesssim {}& T^{\mu(1/8)}\| \PN u \|_{\mathcal{S}^0}\| \PN v \|_{\mathcal{S}^{-p/2}}  \sum_{\smath{1\lesssim N_1\lesssim N}}N_1^{-1/2}, \nonumber
	\end{align}
	where in the last line was used the following inequality (see \cite{MV}): for any $w \in \mathcal{S}^0$ with compact support in $[-T,T]$
	$$
	\|P_N w\|_{L^2_t H^{3/4}} \lesssim \|P_N w\|_{X_p^{0,3/8, 2}} \lesssim T^{\mu(1/8)}\|P_N w\|_{X^{0,1/2, 2}_p}\lesssim T^{\mu(1/8)}\|P_N w\|_{\mathcal{S}^{0}},
	$$
	that can be proved using \eqref{dsg}.
	\\
	ii) $L_{max}= L_1$.
	
	In the first sum in \eqref{est_bil10} we have
	\begin{align}
	\, \lesssim {}& \sum_{\smath{1\lesssim N_1\lesssim N}}\! \cjap{N}^{-p/2} N\|\PNU Q_{N^2\!N_1}u\, \PN v\|_{\!_{L^1_tL^2_x}}\lesssim \sum_{\smath{1\lesssim N_1\lesssim N}}\! \cjap{N}^{-p/2} N\|\PNU Q_{N^2\!N_1}u\|_{\!_{L^2_tL^{\infty}_x}} \|\PN v\|_{\!_{L^2}} \nonumber\\
	\lesssim {}& \sum_{\smath{1\lesssim N_1\lesssim N}}\! N^{-p/2+1} N_1^{1/2}\|Q_{N^2\!N_1}\PNU u\|_{\!_{L^2_tL^{2}_x}} \|\PN v\|_{\! \mathcal{S}\pow{-p/2}}     \nonumber\\
	\lesssim {}& \sum_{\smath{1\lesssim N_1\lesssim N}}\!N_1^{-p/2+3/4} \|P_{N_1}u\|_{\!_{L^2_tH^{3/4}_x}} \|\PN v\|_{\! \mathcal{S}\pow{-p/2}} \nonumber\\
	\lesssim {}&  T^{\mu(1/8)}\| \PN u \|_{\mathcal{S}^0}\| \PN v \|_{\mathcal{S}^{-p/2}}  \sum_{\smath{1\lesssim N_1\lesssim N}}N_1^{-p/2+3/4}. \nonumber
	\end{align}
	The second sum is estimated similarly and the last sum was treated in the case i).
	\\
	iii) $L_{max}= L_2$.
	In this case is suffice to estimate the first sum (the other cases follows of the above cases i) e ii)) 
	\begin{align}
	\sum_{\smath{1\lesssim N_1\lesssim N}}\,\, \sum_{\ \smath{L_2\!\sim\! N^2N_1}}\!\normaNp{\PN\partial_x(\PNU u\,\PN \QLD v)}\lesssim {}& \sum_{\smath{1\lesssim N_1\lesssim N}}N^{-p/2+1}N_1^{1/2}\|\PNU u\|_{\!_{L^2}}\|\PN Q_{N^2\!N_1} v\|_{\!_{L^2}}    \nonumber\\
	\lesssim {}& \sum_{\smath{1\lesssim N_1\lesssim N}}\!N_1^{-p/2+3/4} \|P_{N_1}u\|_{\!_{L^2_tH^{3/4}_x}} \|\PN v\|_{\! \emph{S}\pow{-p/2}} \nonumber\\
	\lesssim {}&  T^{\mu(1/8)}\| \PN u \|_{\mathcal{S}^0}\| \PN v \|_{\mathcal{S}^{-p/2}}. \nonumber
	\end{align}
\end{proof}

\begin{proposition}
	For any $\beta \geq 1 $ there exists  $ 0 < T= T(\beta) < 1 $  such that for any  $ u, v\in\mathcal{Z}_{\beta}$  with compact support in $ [- T, T ] $, we conclude that
	\begin{equation}\label{n-lin-a}
	\| \mathcal{L}\partial_{x}(u v ) \|_{\mathcal{Z}_{\beta}} \lesssim \| u \|_{\mathcal{Z}_{\beta}} \| v \|_{\mathcal{Z}_{\beta}}.
	\end{equation}
\end{proposition}
\begin{proof}
	For $u, v\in\mathcal{Z}_{\beta}$, by definition of the infimum exist $ u_1, v_1\in \mathcal{S}^{- p/2},\;u_2, v_2\in \mathcal{S}^0$ such that 
	\begin{eqnarray}
	u = & u_1 + u_2 , \\
	v = & v_1 + v_2 ,
	\end{eqnarray}
	satisfying 
	\begin{equation}\label{aau}
	\| u \|_{\mathcal{Z}_{\beta}} < \| u_1 \|_{\mathcal{S}^{- p/2}} + \dfrac{1}{\beta}\| u_2 \|_{\mathcal{S}^{0}} \leq 2 \| u \|_{\mathcal{Z}_{\beta}},
	\end{equation}
	and
	\begin{equation}\label{aav}
	\| v \|_{\mathcal{Z}_{\beta}} < \| v_1 \|_{\mathcal{S}^{- p/2}} + \dfrac{1}{\beta}\| v_2 \|_{\mathcal{S}^{0}} \leq 2 \| v \|_{\mathcal{Z}_{\beta}}.
	\end{equation}
	
	Moreover, 
	\begin{equation}\label{aauv}
	\| \mathcal{L}\partial_{x}( u v ) \|_{\mathcal{Z}_{\beta}} \lesssim \| \mathcal{L}\partial_{x}( u_1 v_1 ) \|_{\mathcal{Z}_{\beta}} + \| \mathcal{L}\partial_{x}( u_1 v_2 + u_2 v_1) \|_{\mathcal{Z}_{\beta}} + \| \mathcal{L}\partial_{x}( u_2 v_2 ) \|_{\mathcal{Z}_{\beta}},
	\end{equation}
	where $\mathcal{L} $  is defined in Proposition~\ref{prop-b}.\\
	
	Because $ || u ||_{\mathcal{Z}_{\beta}} \sim || u ||_{\mathcal{S}^{- p/2}},\quad  \text{for}\ u\in \mathcal{S}^{- p/2}$, from \eqref{aauv}, we obtain that
	\begin{equation}\label{aaluv}
	\begin{array}{ll}
	\| \mathcal{L}\partial_{x}( u v ) \|_{\mathcal{Z}_{\beta}} & \lesssim  \| \mathcal{L}\partial_{x}( u_1 v_1 ) \|_{\mathcal{S}^{- p/2}} + \| \mathcal{L}\partial_{x}( u_1 v_2 + u_2 v_1) \|_{\mathcal{S}^{- p/2}} + \| \mathcal{L}\partial_{x}( u_2 v_2 ) \|_{\mathcal{S}^{- p/2}},\\
	&  \lesssim (I) + (II) + (III).
	\end{array}
	\end{equation}
	Now let's estimate each term on the right-hand side. Applying the results \eqref{op} and \eqref{est_bil}, we obtain
	\begin{equation} \label{Iuv}
	\begin{array}{ll}
	\quad (I) & \lesssim \| \partial_{x}(u_1 v_1)\|_{\mathcal{N}^{- p/2}} \lesssim  \| u_1\|_{\mathcal{S}^{- p/2}} \| v_1\|_{\mathcal{S}^{- p/2}},\\
	\quad     & \lesssim \| u_1 \|_{\mathcal{Z}_{\beta}}\| v_1 \|_{\mathcal{Z}_{\beta}}.
	\end{array}
	\end{equation}
	In the third term on the right-hand side, applying  \eqref{nao-lin}, we get
	\begin{equation} \label{IIIuv}
	\begin{array}{ll}
	\quad (III) & \lesssim \| \partial_{x}(u_2 v_2)\|_{\mathcal{N}^{- p/2}} \lesssim  T^{\nu} \| u_2\|_{\mathcal{S}^{0}} \| v_2\|_{\mathcal{S}^{- p/2}},\\
	\quad         & \lesssim  T^{\nu} \| u_2\|_{\mathcal{S}^{0}} \| v_2\|_{\mathcal{S}^{0}}.
	\end{array}
	\end{equation}
	By \eqref{aau}, \eqref{aav} and \eqref{IIIuv}, we have  that
	\begin{equation} \label{IIIuva}
	\begin{array}{ll}
	\quad (III)  & \lesssim  T^{\nu}4 \beta^2  \| u\|_{\mathcal{Z}_{\beta}} \| v\|_{\mathcal{Z}_{\beta}},\\
	\quad         & \lesssim \| u\|_{\mathcal{Z}_{\beta}} \| v\|_{\mathcal{Z}_{\beta}},
	\end{array}
	\end{equation}
	if \quad $ 0 < T \leq \beta^{- 2/\nu} < 1$.
	
	In the second term on the right-hand side, we have that
	\begin{equation} \label{IIuvb}
	\begin{array}{ll}
	(II)  & \lesssim \| \partial_{x}(u_1 v_2)\|_{\mathcal{N}^{- p/2}} + \| \partial_x( u_2v_1)\|_{\mathcal{N}^{- p/2}},\\
	\quad & \lesssim T^{\nu} 4\beta \| u\|_{\mathcal{Z}_{\beta}} \| v\|_{\mathcal{Z}_{\beta}} \\
	\quad & \lesssim \| u\|_{\mathcal{Z}_{\beta}} \| v\|_{\mathcal{Z}_{\beta}},
	\end{array}
	\end{equation}
	if \quad $ 0 < T \leq \beta^{- 1/\nu} < 1$.
	
\end{proof}

We define the operator
\[ F^{T}_{u_0} : u\in\mathcal{Z}_{\beta}\, \mapsto \,\eta(t)W_p(t,\,t)u_0 - \eta(t)\mathcal{L} \partial_{x}(\eta_T u)^2.  \]
We will show that the operator $ F^{T}_{u_0} $ is a contraction on the closed ball $B_{R}:= \left\{ w\in\mathcal{Z}_{\beta}\,:\, \| w \|_{\mathcal{Z}_{\beta}} \leq R \right\} $.

Let $u_0 \in H^{-p/2}$ and $\epsilon>0$, we make the following decomposition
$$
u_0= P_{\lesssim N} u_0+P_{\gg N} u_0,
$$
where $N$ is a dyadic number that we will choose later. By Proposition \ref{prop-a}we have,
$$
\|\eta(t)W_p(t,t)P_{\gg N} u_0 \|_{\mathcal{Z}_{\beta}} \sim \|\eta(t)W_p(t,t)P_{\gg N} u_0 \|_{\mathcal{S}^{-p/2}}\lesssim \|P_{\gg N} u_0 \|_{H^{-p/2}}\leq \epsilon,
$$
where $N=N(\epsilon)$  is large enought.
Also
\begin{equation}
\begin{split}
\|\eta(t)W(t)P_{\lesssim N} u_0 \|_{\mathcal{Z}_{\beta}} \leq&
\frac{1}{\beta} \|\eta(t)W_p(t,t)P_{\lesssim N} u_0 \|_{\mathcal{S}^{0}}\\
\lesssim & \frac{N^{p/2}}{\beta} \|\eta(t)W_p(t,t)P_{\lesssim N} u_0 \|_{\mathcal{S}^{-p/2}}
\\
\lesssim & \frac{N^{p/2}}{\beta} \|P_{\lesssim N} u_0 \|_{H^{-p/2}}
\\
\lesssim & \frac{N^{p/2}}{\beta} \|u_0 \|_{H^{-p/2}}\leq \epsilon,
\end{split}
\end{equation}
if 
$$
\beta \gtrsim  \frac{N^{p/2}}{\epsilon} \|u_0 \|_{H^{-p/2}}.
$$
Thus, for $u \in B_R$,  we obtain
\begin{equation}
\begin{split}
\| F^{T}_{u_0}u \|_{\mathcal{Z}_{\beta}} \leq &  C\epsilon+C \|\eta_T u \|_{\mathcal{Z}_{\beta}}^2\\
\leq &\dfrac{R}{2}+CR^2\\
\leq & R,
\end{split}
\end{equation}
where was considered $R=2C \epsilon$, $0<\epsilon < 1/(4C^2)$. Also,
\begin{equation}
\begin{split}
\| F^{T}_{u_0}u_1- F^{T}_{u_0}u_2 \|_{\mathcal{Z}_{\beta}} \leq &  C \|\eta_T (u_1+u_2) \|_{\mathcal{Z}_{\beta}} \|\eta_T (u_1-u_2) \|_{\mathcal{Z}_{\beta}}\\
\leq &2CR \| (u_1-u_2) \|_{\mathcal{Z}_{\beta}}\\
\leq & 4C^2 \epsilon \| (u_1-u_2) \|_{\mathcal{Z}_{\beta}},
\end{split}
\end{equation}
and taking $4C^2 \epsilon <1$, $F^{T}_{u_0}$ is a contraction . By standard arguments, the uniqueness holds in the space $\mathcal{S}^{-p/2}_T$ endowed with the norm
$$
\|u\|_{\mathcal{S}^{-p/2}_T}:= \inf_{v \in \mathcal{S}^{-p/2}} \{ \| v \|_{\mathcal{S}^{-p/2}}, \quad v \equiv u \,\, \textrm{on} \,\,  ]0, T[ \}.
$$

%%%%%%%%%%%%%%%%%%%%%%%%%%%%%FIM well posedness%%%%%%%%%%%%%%%%%%%%%%%%

\section{\sc Ill-posedness results}
\begin{lemma}\label{lem1}
	Let $g: \mathbb{R}^n \to  \mathbb{R}$ be a continuous function and $f: \mathbb{R}^n \to  \mathbb{R}$ be a positive function. If any $x \in  \mathbb{R}^n$, $|g(x)| \geq c_0 \geq 0$, then 
	\begin{equation}\label{eq0}
	\left| \int_{ \mathbb{R}^n} f(x) g(x) dx    \right| \geq c_0   \int_{ \mathbb{R}^n} f(x) dx.
	\end{equation}
\end{lemma}
\begin{remark}
	Observe that the estimate \eqref{eq0} in Lemma \ref{lem1} is false if $g$ is a complex valued function. In fact, if we consider $n=1$, $g(x)=e^{ix}$ and $f(x)=\chi_{[-\pi, \pi]}(x)$ the hypotheses of Lemma  \ref{lem1} are satisfied but the estimate  \eqref{eq0} does not hold. Also  \eqref{eq0} is false if $f$ is a positive function such that there exist $\xi$ with $|\widehat{f} (\xi)| <\widehat{f} (0)$ and $g(x)=e^{-ix\xi}$.
\end{remark}
Let $N \gg 1$, $I_N=[N,N+2]$, $\alpha=p/2$ and
\begin{equation}\label{eq1}
\widehat{\Phi_N} (\xi)=N^{\alpha}(\chi_{I_N}(\xi)+\chi_{I_N}(-\xi)),
\end{equation}
then 
\begin{equation}\label{eq2}
\|\Phi_N\|_{H^{-p/2}}^2=2N^{2\alpha} \int_{|\xi| \sim N}\langle \xi  \rangle^{-p} d\xi \sim 1, 
\end{equation}
and 
\begin{equation}\label{eq3}
\|\Phi_N\|_{H^{s}}^2=2N^{2\alpha} \int_{|\xi| \sim N}\langle \xi  \rangle^{ } d\xi \sim N^{p + 2s},
\end{equation}
then $ \Phi_N \to 0$ in $H^s$ if $ s<-\frac{p}2$.

For $t>0$ we define
\begin{equation}\label{eq4}
A_2(t,h,h)= \int_0^t S_p(t-t')\partial_x (S_p(t')h)^2 dt', 
\end{equation}
taking Fourier transform we have
\begin{equation}\label{eq5}
\mathcal{F}_x(A_2(t,\Phi_N,\Phi_N))(\xi)=i \xi e^{-t|\xi|^p +it\xi^3} \int_{\mathbb{R}}\widehat{\Phi_N} (\xi_1)\widehat{\Phi_N} (\xi- \xi_1)\left(\dfrac{ e^{t \varphi}-1 }{\varphi}  \right) d\xi_1,
\end{equation}
where $\varphi=\varphi_1+i \varphi_2:=|\xi|^p-|\xi_1|^p-|\xi-\xi_1|^p +i(-\xi^3 +\xi_1^3+(\xi-\xi_1)^3)$,
and consequently
\begin{equation}\label{eq6}
\|A_2(t,\Phi_N,\Phi_N)\|_{H^s}^2=\int_{\mathbb{R}}    N^{2p}  e^{-2t|\xi|^p}  \langle  \xi  \rangle^{ 2s}|\xi|^2 \left|  \int_{K_{\xi}}\dfrac{ e^{t \varphi}-1 }{\varphi} d\xi_1 \right|^2  d\xi,
\end{equation}
where
$$
K_{\xi}=\{\xi_1  \,/   \, \xi-\xi_1 \in I_N, \quad  \xi_1 \in -I_N  \}\cup \{\xi_1   \,/ \,   \xi-\xi_1 \in -I_N, \quad  \xi_1 \in I_N  \}.
$$
We note that if $|\xi| \leq 1/2$, then $|K_{\xi}| \geq 1$ and $\xi_1 \in K_{\xi}$ implies $\varphi_2=3\xi \xi_1(\xi-\xi_1 ) \sim -N^2 \xi $.
Let
\begin{equation}\label{eq7}
\begin{split}
 \left| e^{-t|\xi|^p} \int_{K_{\xi}}\dfrac{ e^{t \varphi}-1 }{\varphi} d\xi_1 \right|
= &  \left| \int_{K_{\xi}}\dfrac{ e^{t(-|\xi_1|^p-|\xi-\xi_1|^p)-3it \xi \xi_1(\xi- \xi_1) }-e^{-t|\xi|^p} }{\varphi_1+i \varphi_2} d\xi_1 \right|\\
:= &  \left| \int_{K_{\xi}}\dfrac{f}{g}\, d\xi_1 \right|\\
\geq &  \left| \int_{K_{\xi}} \textrm{Re} \,\left\{ \dfrac{f}{g}\right\}\, d\xi_1 \right|.
\end{split}
\end{equation}
Observe that
\begin{equation}\label{eq70}
\left| \textrm{Re} \,\left\{ \dfrac{f}{g}\right\}\right|=\dfrac{|  \textrm{Re} \,f  \, \textrm{Re} \,g+   \textrm{Im} \,f \,  \textrm{Im} \,g  |}{|g|^2}\geq \dfrac{|  \textrm{Re} \,f  \, \textrm{Re} \,g |}{|g|^2}-\dfrac{| \textrm{Im} \,f |}{|g|}.
\end{equation}
For $0<\gamma \ll 1$, $N\gg 1$, $|\xi|\sim   \gamma$ and $\xi_1 \in K_{\xi}$, one can obtain
\begin{equation}\label{eq8}
\begin{split}
\textrm{Re} \,f= \textrm{Re} \,\{ e^{t(-|\xi_1|^p-|\xi-\xi_1|^p)-3it \xi \xi_1(\xi- \xi_1)  }-e^{-t|\xi|^p}\} \leq &  e^{t(-|\xi_1|^p-|\xi-\xi_1|^p)  }-e^{-t\gamma^p}\\
\leq & e^{-ctN^p  }-e^{-t\gamma^p}\\
\leq & \dfrac{-e^{-t\gamma^p}}2,
\end{split}
\end{equation}
\begin{equation}\label{eq9}
\begin{split}
| \textrm{Re} \,g|=|\varphi_1|= |\, |\xi|^p -|\xi_1|^p-|\xi-\xi_1|^p\, |\sim N^p, \quad | \textrm{Im} \,g|=|\varphi_2|\sim N^2 |\xi|
\end{split}
\end{equation}
and also 
\begin{equation}\label{eq10}
\begin{split}
|\textrm{Im} \,f|= |\textrm{Im} \,\{ e^{t(-|\xi_1|^p-|\xi-\xi_1|^p)-3it \xi \xi_1(\xi- \xi_1)  }-e^{-t|\xi|^p}\}| \leq e^{-ctN^p}.
\end{split}
\end{equation}
using \eqref{eq70}-\eqref{eq10} and $N \gg 1$, it follows that
\begin{equation}\label{eq11}
\left| \textrm{Re} \,\left\{ \dfrac{f}{g}\right\}\right| \gtrsim \dfrac{e^{-t\gamma^p} N^p}{ N^{2p}+\gamma^2N^4}-  \dfrac{e^{-ct N^p}}{ N^{p}+\gamma N^2} \gtrsim  \dfrac{e^{-t\gamma^p} N^p}{ N^{2p}+\gamma^2N^4}.
\end{equation}
Considering $\gamma \sim 1$,  $p \geq 2$ and combining \eqref{eq6}, \eqref{eq7} and \eqref{eq11} we have
\begin{equation}\label{eq12}
\begin{split}
\|A_2(t,\Phi_N,\Phi_N)\|_{H^s}^2 \gtrsim & \int_{|\xi|\sim \gamma}    N^{2p}   \langle  \xi  \rangle^{2s }|\xi|^2 \left|\dfrac{e^{-t\gamma^p} N^p}{ N^{2p}+\gamma^2N^4}\right|^2  d\xi\\
\gtrsim &    N^{2p}  \dfrac{e^{-2t\gamma^p} N^{2p}}{ N^{4p}+\gamma^4N^8}  \int_{|\xi|\sim \gamma} \langle  \xi  \rangle^{2s }|\xi|^2  d\xi \\
\gtrsim &      \dfrac{\gamma^3 e^{-2t\gamma^p} N^{4p}}{ N^{4p}+\gamma^4N^8}.
\end{split}
\end{equation}
Therefore
\begin{equation}\label{1eq12}
\begin{split}
\|A_2(t,\Phi_N,\Phi_N)\|_{H^s}\geq     C_0>0.
\end{split}
\end{equation}

By Theorem \ref{teo1}, there exist $T>0$ and $0<\epsilon_0\ll 1$ such that for any $|\epsilon| \leq \epsilon_0$, any $\|h\|_{H^{-p/2}} \leq 1$ and $t \in [0,T]$,
\begin{equation}\label{eq13}
u(t, \epsilon h)= \epsilon S_p(t)h+ \sum_{k=2}^{\infty} \epsilon^k A_k(t, h^k),
\end{equation}
where $h^k:=(h, \dots, h)$, $h^k \to A_k(t, h^k)$ is a k-linear continuous map from $\left[H^{-p/2}(\R)\right]^k$ into $C([0,T]; \, H^{-p/2}(\R)\,)$ and the series converges absolutely in $C([0,T]; \, H^{-p/2}(\R)\,)$.
If $h=\Phi_N$, from \eqref{eq13} we get
\begin{equation}\label{eq14}
u(t, \epsilon \Phi_N)-\epsilon^2 A_2 (t, \Phi_N, \Phi_N)= \epsilon S(t) \Phi_N+ \sum_{k=3}^{\infty} \epsilon^k A_k(t,  \Phi_N^k),
\end{equation}
if $k\geq 3$, then $| \epsilon|^k=| \epsilon|^3| \epsilon|^{k-3 } \leq | \epsilon|^3| \epsilon_0|^{k-3 }$, thus using \eqref{eq2} we obtain
\begin{equation}\label{eq15}
\begin{split}
\|\sum_{k=3}^{\infty} \epsilon^k A_k(t,  \Phi_N^k)\|_{H^{-p/2}}\leq  | \epsilon|^3 \sum_{k=3}^{\infty}| \epsilon_0|^{k-3 }\| A_k(t,  \Phi_N^k)\|_{H^{-p/2}}\lesssim  | \epsilon|^3 \sum_{k=3}^{\infty}| \epsilon_0|^{k-3 }\| \Phi_N\|_{H^{-p/2}}^k\lesssim  | \epsilon|^3. 
\end{split}
\end{equation}
Now combining \eqref{eq3}, \eqref{eq14} and \eqref{eq15} we conclude that, for $s<-\frac{p}2$,
\begin{equation}\label{eq16}
\begin{split}
\epsilon^2 \| A_2 (t, \Phi_N, \Phi_N)  \|_{H^{s}}- \|   u(t, \epsilon \Phi_N) \|_{H^{s}} \leq  \left\|   u(t, \epsilon \Phi_N)-\epsilon^2 A_2 (t, \Phi_N, \Phi_N)   \right\|_{H^{s}} \leq  C | \epsilon|^3+C|\epsilon| N^{s+p/2},
\end{split}
\end{equation}
hence
\begin{equation}\label{eq17}
\begin{split}
\|   u(t, \epsilon \Phi_N) \|_{H^{s}} \geq \epsilon^2 \| A_2 (t, \Phi_N, \Phi_N)  \|_{H^{s}}-  C | \epsilon|^3-C|\epsilon| N^{s+p/2},
\end{split}
\end{equation}
and considering \eqref{1eq12}

\begin{align}
\|   u(t, \epsilon \Phi_N) \|_{H^{s}} &\geq \epsilon^2C_0-  C | \epsilon|^3-C |\epsilon|N^{s+p/2}\geq \epsilon^2(C_0-C|\epsilon|)-C |\epsilon|N^{s+p/2},\nonumber\\
& \geq\frac{3}{4}C_0\epsilon^2-C|\epsilon|N^{s+p/2}, \label{eq18}
\end{align}

if $|\epsilon|<c_0/(4C)$. Since  $u(t,0)=0$ and $ \Phi_N \to 0$ in $H^{s}$ for  $s<-\frac{p}2$ and $p\geq 2$, thus we conclude that the flow-map from $H^s(\R)$ into $H^s(\R)$, for $s<-p/2$ and $p\geq 2$, is discontinuous at the origin by letting $N$ tend to infinity.

Finally, to prove the Remark \ref{c2ill}, it is suffices to follow the steps \eqref{eq1}-\eqref{1eq12} with  $
	\widehat{\Phi_N} (\xi)=N^{\alpha}(\chi_{I_N}(\xi)+\chi_{I_N}(-\xi))
$ for $N \gg 1$, $I_N=[N,N+2]$, $\alpha>0$ (to be choosen later). Indeed, for $s\leq -\alpha$ $$\|\Phi_N\|_{H^s}\sim N^{s+\alpha}\lesssim 1\ \ \ \ \ \textrm{and}\ \ \ \|A_2(t,\Phi_N,\Phi_n)\|_{H^s}\geq \frac{N^{4\alpha} N^{2p}\gamma^3}{N^{4p}+\gamma^4N^8}.$$
So, taking $\gamma\sim N^{p-2}\leq 1$ (in order to get $N^{4p}\sim \gamma^4N^8$) we obtain $\|A_2(t,\Phi_N,\Phi_n)\|_{H^s}\geq C_0$, if $\alpha=3/2-p/4$ when $0\leq p\leq 2$.

%\vspace{5mm}}
	%The authors 

% ----------------------------------------------------------
% Referências bibliográficas
% ----------------------------------------------------------
%
\renewcommand{\refname}
	
%%%%%%%%%%%%%%%%%%%%%%%%%%%%%%%%%%%%%%%%%%%%%%%%%
% FIM DE REFERÊNCIAS
%%%%%%%%%%%%%%%%%%%%%%%%%%%%%%%%%%%%%%%%%%%%%%%%%

\appendix
\section{Elementary estimates}
\begin{lemma}[Bernstein type estimates]
		\begin{equation}\label{Bernst1}
		\|P_N(fg)\|_{L_x^2} \lesssim N^{1/2}\|f\|_{L_x^2}\|g\|_{L_x^2},
		\end{equation}
		\begin{equation}\label{Bernst2}
		\|P_{N}f \, P_{N_1}g\|_{L_x^2} \lesssim N_1^{1/2}\|P_{N}f\|_{L_x^2}\|P_{N_1}g\|_{L_x^2},
		\end{equation}
		\begin{equation}\label{Bernst3}
		\|P_{N} \left( Q_{L} \, P_{N_1}f \, P_{N_1}g \right)\|_{L^2} \lesssim N^{1/2}L^{1/2}\|P_{N_1}f\|_{L^2}\|P_{N_1}g\|_{L^2},
		\end{equation}
		\begin{equation}\label{Bernst4}
		\|P_{N} Q_{\gtrsim \lambda}f \|_{L^2} \lesssim  \dfrac{1}{\lambda^{1/2}}\|P_{N}f\|_{\mathcal{S}^0}.
		\end{equation}
	\end{lemma}
	\begin{proof}
		In order to prove \eqref{Bernst1}, using Plancherel's identity and properties of Fourier transform
		\begin{equation*}%\label{Bernst1}
		\|P_N(fg)\|_{L_x^2} =\|\varphi_N \|_{L_x^2}\, \| \mathcal{F}_x\{fg\}\|_{L_x^\infty} \lesssim N^{1/2}\|f\|_{L_x^2}\|g\|_{L_x^2},
		\end{equation*}
		The inequality \eqref{Bernst2} is a consequence of properties of convolution and Cauchy-Schwartz inequality:
		\begin{equation*}%\label{Bernst2}
		\|P_{N}f \, P_{N_1}g\|_{L_x^2} \lesssim \|P_{N}f \|_{L_x^2}  \|\mathcal{F}_x\{P_{N_1}g\}\|_{L_x^1} \lesssim N_1^{1/2}\|P_{N}f\|_{L_x^2}\|P_{N_1}g\|_{L_x^2},
		\end{equation*}
To prove \eqref{Bernst3}, let $\xi_2=  \xi- \xi_1$, $\tau_2= \tau - \tau_1$, using Minkowsky's inequality, properties of the Fourier transform and Cauchy-Schwarz two times, we have,
\begin{equation}\label{1Bernst3}
\begin{split}
\|P_{N} \left( Q_{L} \, P_{N_1}f \, P_{N_1}g \right)\|_{L^2} =&\|\varphi_N(\xi) \int_{\R^2}\varphi_{N_1}(\xi_2) \tilde{g}(\xi_2, \tau_2)\psi_{L_1}(\xi_1, \tau_1)\varphi_{N_1}(\xi_1) \tilde{f}(\xi_1, \tau_1) d\xi_1 d \tau_1\|_{L_{\xi, \tau}^2}\\
 \leq & \int_{\R^2}\|\varphi_N(\xi)\varphi_{N_1}(\xi_2) \tilde{g}(\xi_2, \tau_2)\|_{L_{\xi, \tau}^2}| \psi_{L_1}(\xi_1, \tau_1)\varphi_{N_1}(\xi_1) \tilde{f}(\xi_1, \tau_1)| d\xi_1 d \tau_1 \\
 \leq &\|P_{N_1}g\|_{L^2} \int_{\R}\| \psi_{L_1}(\xi_1, \tau_1)\|_{L^2_{\tau_1}}|\varphi_{N_1}(\xi_1)|\,\| \tilde{f}(\xi_1, \tau_1)\|_{L^2_{\tau_1}} d\xi_1 \\
 \leq &\|P_{N_1}g\|_{L^2}L^{1/2} N^{1/2}\|P_{N_1}f\|_{L^2}\|P_{N_1}g\|_{L^2}.
\end{split}
\end{equation} 
Now we will prove \eqref{Bernst4}. Using Cauchy-Schwarz and inequality \eqref{elem03}, we obtain
\begin{equation}\label{1Bernst4}
\begin{split}
\|P_{N} Q_{\gtrsim \lambda}f \|_{L^2} = &\| P_{N_1} (  \sum_{L \gtrsim \lambda} Q_L f )\|_{L^2}\\
\lesssim & \left(  \sum_{L \gtrsim \lambda} L  \|P_{N} Q_{L}f \|_{L^2}^2 
\right)^{1/2} \left(  \sum_{L \gtrsim \lambda} \dfrac{1}{L} 
\right)^{1/2}\\
\lesssim &\dfrac{1}{\lambda^{1/2}}\|P_{N}f\|_{\mathcal{S}^0}.
\end{split}
\end{equation}
\end{proof}
We also have the following localization properties
\begin{equation}\label{locprop}
\|f\|_{\mathcal{S}^{\theta}} \sim \left(  \sum_{N} \|P_N f\|_{\mathcal{S}^{\theta}}^2 \right)^{1/2},\quad \|f\|_{\mathcal{N}^{\theta}} \sim \left(  \sum_{N} \|P_N f\|_{\mathcal{N}^{\theta}}^2 \right)^{1/2},
\end{equation}
\begin{equation*}
\|P_N f\|_{\mathcal{S}^{\theta}} \sim \langle N\rangle^{\theta}\|P_N f\|_{\mathcal{S}^{0}}\,\,, \quad
\|P_N f\|_{X^{\theta,1/2,1}} \sim \langle N\rangle^{\theta}\|P_N f\|_{X^{0,1/2,1}}\,\,,
\end{equation*}
and 
\begin{equation*}
\|P_N f\|_{Y^{\theta,1/2}} \sim \langle N\rangle^{\theta}\|P_N f\|_{Y^{0,1/2}}.
\end{equation*}
\vspace{5mm}
%\section*{Acknowledgments

\end{document}